\documentclass[oneside,english,american]{amsart}
\usepackage[T1]{fontenc}
\usepackage[latin9]{inputenc}
\usepackage{xcolor}
\usepackage{verbatim}
\usepackage{float}
\usepackage{amstext}
\usepackage{amsthm}
\usepackage{amssymb}
\usepackage{graphicx}
\PassOptionsToPackage{normalem}{ulem}
\usepackage{ulem}

\makeatletter

\floatstyle{ruled}
\newfloat{algorithm}{tbp}{loa}
\providecommand{\algorithmname}{Algorithm}
\floatname{algorithm}{\protect\algorithmname}
\providecolor{lyxadded}{rgb}{0,0,1}
\providecolor{lyxdeleted}{rgb}{1,0,0}

\DeclareRobustCommand{\lyxsout}[1]{\ifx\\#1\else\sout{#1}\fi}

\numberwithin{equation}{section}
\numberwithin{figure}{section}
\theoremstyle{plain}
\newtheorem{thm}{\protect\theoremname}[section]
  \theoremstyle{plain}
  \newtheorem{lem}[thm]{\protect\lemmaname}
  \theoremstyle{remark}
  \newtheorem*{rem*}{\protect\remarkname}
  \theoremstyle{remark}
  \newtheorem{rem}[thm]{\protect\remarkname}

\usepackage[T1]{fontenc}
\usepackage{ae,aecompl}

\makeatother

\usepackage{babel}
  \addto\captionsamerican{\renewcommand{\lemmaname}{Lemma}}
  \addto\captionsamerican{\renewcommand{\remarkname}{Remark}}
  \addto\captionsamerican{\renewcommand{\theoremname}{Theorem}}
  \addto\captionsenglish{\renewcommand{\lemmaname}{Lemma}}
  \addto\captionsenglish{\renewcommand{\remarkname}{Remark}}
  \addto\captionsenglish{\renewcommand{\theoremname}{Theorem}}
  \providecommand{\lemmaname}{Lemma}
  \providecommand{\remarkname}{Remark}
\addto\captionsenglish{\renewcommand{\algorithmname}{Algorithm}}
\providecommand{\theoremname}{Theorem}

\begin{document}

\title[On computing distributions of products of random variables]{On computing distributions of products of non-negative independent
random variables  }

\author{Gregory Beylkin, Lucas Monz\'{o}n and Ignas Satkauskas }

\address{Department of Applied Mathematics \\
 University of Colorado at Boulder \\
 UCB 526 \\
 Boulder, CO 80309-0526 }
\begin{abstract}
We introduce a new functional representation of probability density
functions (PDFs) of non-negative random variables via a product of
a monomial factor and linear combinations of decaying exponentials
with complex exponents. This approximate representation of PDFs is
obtained for any finite, user-selected accuracy. Using a fast algorithm
involving Hankel matrices, we develop a general numerical method for
computing the PDF of the sums, products, or quotients of any number
of non-negative independent random variables yielding the result in
the same type of functional representation. We present several examples
to demonstrate the accuracy of the approach.
\end{abstract}

\keywords{product of non-negative independent random variables, probability
density function}
\maketitle

\section{Introduction}

Consider two non-negative independent random variables $X$ and $Y$
with probability density functions (PDFs) $f$ and $g$. It is well
known that the PDFs of their sum, $X+Y$, is given by the convolution
\begin{equation}
s\left(t\right)=\int_{0}^{t}f\left(t-y\right)g(y)dy,\label{eq:PDF of the sum}
\end{equation}
the PDF $p$ of their product, $XY$, is given by
\begin{equation}
p\left(t\right)=\int_{0}^{\infty}\int_{0}^{\infty}f\left(x\right)g\left(y\right)\delta\left(xy-t\right)dxdy,\,\,t\ge0,\label{eq:integral-to-compute-pos}
\end{equation}
where $\delta$ is the delta function, or alternatively, as 
\begin{equation}
p\left(t\right)=\int_{0}^{\infty}f\left(x\right)g\left(t/x\right)\frac{1}{x}dx,\label{eq:integral traditional form}
\end{equation}
and the PDF $q$ of their quotient, $X/Y$, is given by
\begin{eqnarray}
q\left(t\right) & = & \int_{0}^{\infty}\int_{0}^{\infty}f\left(x\right)g\left(y\right)\delta\left(\frac{x}{y}-t\right)dxdy\nonumber \\
 & = & \int_{0}^{\infty}f\left(ty\right)g\left(y\right)ydy,\,\,t\ge0.\label{eq:integral to compute quotient}
\end{eqnarray}
In this paper we introduce a new approximate representation of PDFs
of non-negative random variables via a product of a monomial factor
and a linear combination of decaying exponentials with complex exponents.
Importantly, representing PDFs in this form allows us to evaluate
these integrals numerically so that the resulting PDFs have the same
functional representation as the original PDFs and, thus, can be used
in further computations. Essentially, we provide algorithms to represent
the PDF of a non-negative random variable within any user-selected
accuracy via an optimal linear combination of Gamma-like distributions
with a common shape parameter and possibly a complex valued rate parameter
(with a negative real part). By optimal we understand a linear combination
with minimal number of terms for a given accuracy. We note that while
in principle it is possible to use a representation with only a linear
combination of decaying and oscillatory exponentials, introducing
an additional monomial factor to account for a possible rapid change
of the PDF near zero makes the approximation significantly more efficient. 

Among the operations on random variables mentioned above, computing
the PDF of their product is particularly difficult. It is well known
(see e.g.~\cite{SPRING:1979}) that the Mellin transform of $p$
in \eqref{eq:integral traditional form} is equal to the product of
the Mellin transforms of $f$ and $g$ (the function $p$ is the so-called
Mellin convolution of $f$ and $g$). However, numerical implementation
of the Mellin transform has not resulted in a reliable numerical method.
The only universal method currently available for computing the PDF
of the product of two non-negative independent random variables relies
on a Monte-Carlo type approach, where one samples the individual PDFs,
computes their products, and collects enough samples to achieve certain
accuracy in the computation of $p$ in \eqref{eq:integral traditional form}.
However, due to the slow convergence of such methods (typically $1/\sqrt{N}$,
where $N$ is the number of samples) achieving high accuracy is not
feasible. 

Over the years, for particular PDFs $f$ and $g$, there have been
a number of results showing that $p$ may be computed using special
functions or a series expansion \cite{EPSTEI:1948,SPR-THO:1966,LOMNIC:1967,NAD-KOT:2006,C-K-L-C:2012,Z-W-H-T:2012,SHA-KIB:2007}.
While such results are appropriate for specific distributions, they
do not provide a universal method to compute $p$.

Our representation of PDFs of non-negative random variables differs
significantly from the one we developed for random variables on the
real line in \cite{BE-MO-SA:2017}. For real-valued random variables
with smooth PDFs (except possibly at a finite number of points where
they can have integrable singularities), we use the approximation
via multiresolution Gaussian mixtures in \cite{BE-MO-SA:2017}. In
contrast, our new representation is tailored to non-negative random
variables and accounts for the boundary point (that is, zero) near
which PDFs can be rapidly changing. 

Our approach relies on several algorithms to construct, for a given
accuracy, a (near) optimal representation of functions via a linear
combination of exponentials. These algorithms have their mathematical
foundation in the seminal AAK theory for optimal rational approximations
in the infinity norm \cite{AD-AR-KR:1968,AD-AR-KR:1968a,AD-AR-KR:1971}.
This theory relies on properties of infinite Hankel matrices (Hankel
operators) constructed from the functions to be approximated. Practical
algorithms use finite Hankel matrices and their singular value decomposition
as in \cite{KU-AR-BH:1983,HUA-SAR:1988,HUA-SAR:1990,HUA-SAR:1991}
or (a related) con-eigenvalue decomposition as in \cite{BEY-MON:2005,HAU-BEY:2012}.
These algorithms effectively use analysis-based approximations rather
than a straightforward optimization and are well suited for our purposes.

We introduce our representation for PDFs and derive PDFs for sums,
products and quotients of two random variables in this representation
in Section~\ref{sec:Representation-of-distributions}. In Section~\ref{sec:Algorithm-for-computing}
we briefly describe algorithms we use for computing a near optimal
representation of functions via a linear combination of exponentials
as well as a fast algorithm for computing the Singular Value Decomposition
(SVD) of a low rank Hankel matrix. We illustrate our approach by numerical
examples presented in Section~\ref{sec:Examples-of-computing} and,
in Section~\ref{sec:Computing-expectations-of}, we show that expectations
of functions of non-negative random variables are easily evaluated
using our new representation of their PDFs. Finally, we briefly discuss
further work in Section~\ref{sec:Conclusions-and-further}.

\section{Representation of PDFs of sums, products and quotients of non-negative
independent random variables\label{sec:Representation-of-distributions}}

For a user-selected accuracy $\epsilon$, we approximate the PDF $f_{X}$
of a non-negative random variable $X$ as 
\begin{equation}
\left|f_{X}\left(x\right)-f\left(x\right)\right|\le\epsilon,\,\,\,x>0,\label{eq:approximation-error}
\end{equation}
where 
\begin{equation}
f\left(x\right)=x^{\alpha-1}\sum_{m=1}^{M}a_{m}e^{-\xi_{m}x},\,\,\,\mathcal{R}e\left(\xi_{m}\right)>0,\,\,\,\,\alpha>0,\label{eq:form to maintain}
\end{equation}
and $M$ is as small as possible. Given two non-negative independent
random variables $X$ with PDF \eqref{eq:form to maintain} and $Y$
with PDF 
\begin{equation}
g\left(y\right)=y^{\beta-1}\sum_{n=1}^{N}b_{n}e^{-\eta_{n}y},\,\,\,\mathcal{R}e\left(\eta_{m}\right)>0,\,\,\,\,\beta>0,\label{eq:second PDF}
\end{equation}
we demonstrate that the PDFs of their sum $X+Y$, product $XY$ and
quotient $X/Y$ can be represented in the same functional form thus
enabling a numerical calculus of non-negative random variables. We
note that while in some cases it may be possible to avoid using the
explicit factor $x^{\alpha-1}$ in \eqref{eq:form to maintain}, this
factor significantly reduces the number of terms required if $f$
has a rapid change near the origin. 

In this section we derive formulas (in terms of special functions)
for the results of these operations on independent random variables
with PDFs of the form \eqref{eq:form to maintain} and \eqref{eq:second PDF}.
Then, in Section~\ref{sec:Algorithm-for-computing}, we show how
to obtain the approximation of $f_{X}$ (and similarly, of $g_{Y}$)
in our desired functional form \eqref{eq:form to maintain} and how
to convert the results of operations on them back into the same functional
form.

We start by deriving the PDF of the sum $s$ of two non-negative independent
random variables $X$ and $Y$. We show that
\begin{lem}
\label{lem:1}The PDFs in \eqref{eq:PDF of the sum} of the sum of
two non-negative independent random variables $X$ and $Y$, with
PDFs $f$ and $g$ in \eqref{eq:form to maintain} and \eqref{eq:second PDF},
can be written as
\[
s\left(t\right)=t^{\alpha+\beta-1}r\left(t\right),
\]
where 
\begin{equation}
r\left(t\right)=\frac{\Gamma\left(\alpha\right)\Gamma\left(\beta\right)}{\Gamma\left(\alpha+\beta\right)}\sum_{m=1}^{M}a_{m}e^{-\xi_{m}t}\sum_{n=1}^{N}b_{n}\,_{1}F_{1}\left(\beta,\alpha+\beta;\left(\xi_{m}-\eta_{n}\right)t\right),\label{eq:function to approximate for the sum}
\end{equation}
$\Gamma$ is the gamma function and $\,_{1}F_{1}$ is the confluent
hypergeometric function,
\[
\,_{1}F_{1}\left(a,b;z\right)=1+\frac{a}{b}\frac{z}{1!}+\frac{a\left(a+1\right)}{b\left(b+1\right)}\frac{z^{2}}{2!}+\frac{a\left(a+1\right)\left(a+2\right)}{b\left(b+1\right)\left(b+2\right)}\frac{z^{3}}{3!}\dots.
\]
\end{lem}

\begin{proof}
The result follows using \cite[Formula 3.383.1]{GRA-RYZ:2015},
\begin{eqnarray*}
s\left(t\right) & = & \int_{0}^{t}\left(t-y\right)^{\alpha-1}\left(\sum_{m=1}^{M}a_{m}e^{-\xi_{m}\left(t-y\right)}\right)y^{\beta-1}\left(\sum_{n=1}^{N}b_{n}e^{-\eta_{n}y}\right)dy\\
 & = & \sum_{m=1}^{M}\sum_{n=1}^{N}a_{m}b_{n}e^{-\xi_{m}t}\int_{0}^{t}\left(t-y\right)^{\alpha-1}y^{\beta-1}e^{\left(\xi_{m}-\eta_{n}\right)y}dy\\
 & = & t^{\alpha+\beta-1}r\left(t\right),
\end{eqnarray*}
and replacing the beta function $B\left(x,y\right)$ in \cite[Formula 3.383.1]{GRA-RYZ:2015}
via its functional relation with the gamma function \cite[Formula 8.384.1]{GRA-RYZ:2015},
\[
B\left(x,y\right)=\frac{\Gamma\left(x\right)\Gamma\left(y\right)}{\Gamma\left(x+y\right)}.
\]
\end{proof}
\begin{rem*}
The fact that the sum is independent of the order of $X$ and $Y$,
$p_{X+Y}=p_{Y+X}$, follows from Kummer's first transformation of
$\,_{1}F_{1}$ \cite[p. 191]{AN-AS-RO:1999}. %
\end{rem*}
Next we show how to compute the PDF of the product $p$ of two non-negative
independent random variables $X$ and $Y$. 
\begin{lem}
\label{lem:2}The PDF $p$ in \eqref{eq:integral traditional form}
of the product of two non-negative independent random variables $X$
and $Y$, with PDFs $f$ and $g$ in \eqref{eq:form to maintain}
and \eqref{eq:second PDF}, can be written as 
\begin{equation}
p\left(t\right)=p_{\left(\alpha+\beta-2-\left|\alpha-\beta\right|\right)/2}\left(t\right)=t^{\left(\alpha+\beta-2-\left|\alpha-\beta\right|\right)/2}v\left(t\right),\label{eq:final form}
\end{equation}
where 
\begin{equation}
v\left(t\right)=2\sum_{m=1}^{M}\sum_{n=1}^{N}a_{m}b_{n}\left(\frac{\eta_{n}}{\xi_{m}}\right)^{\left(\alpha-\beta\right)/2}t^{\left|\alpha-\beta\right|/2}K_{\left|\alpha-\beta\right|}\left(2\sqrt{t\xi_{m}\eta_{n}}\right)\label{eq:function_q}
\end{equation}
and $K$ is the modified Bessel function of the second kind. 
\end{lem}

\begin{proof}
We rewrite \eqref{eq:integral traditional form} as 
\begin{equation}
p\left(t\right)=t^{\beta-1}\sum_{m=1}^{M}\sum_{n=1}^{N}a_{m}b_{n}\int_{0}^{\infty}x^{\alpha-\beta-1}e^{-\xi_{m}x}e^{-\eta_{n}\frac{t}{x}}dx.\label{eq:to compute}
\end{equation}
From \cite[3.471.9]{GRA-RYZ:2015} and using that $K_{-\nu}=K_{\nu}$,
we have 
\[
\int_{0}^{\infty}x^{\nu-1}e^{-\left(\xi x+\eta t/x\right)}dx=2\left(\frac{\eta t}{\xi}\right)^{\frac{\nu}{2}}K_{\left|\nu\right|}\left(2\sqrt{t\eta\xi}\right),\,\,\,\mathcal{R}e\left(\xi\right)>0\,\,\,\mbox{and}\,\,\,\mathcal{R}e\left(\eta\right)>0,
\]
where, in our case, $\nu=\alpha-\beta$. We thus obtain
\begin{eqnarray*}
p\left(t\right) & = & 2t^{\beta-1}\sum_{m=1}^{M}\sum_{n=1}^{N}a_{m}b_{n}\left(\frac{\eta_{n}t}{\xi_{m}}\right)^{\left(\alpha-\beta\right)/2}K_{\left|\alpha-\beta\right|}\left(2\sqrt{t\xi_{m}\eta_{n}}\right)\\
 & = & 2t^{\left(\alpha+\beta-2\right)/2}\sum_{m=1}^{M}\sum_{n=1}^{N}a_{m}b_{n}\left(\frac{\eta_{n}}{\xi_{m}}\right)^{\left(\alpha-\beta\right)/2}K_{\left|\alpha-\beta\right|}\left(2\sqrt{t\xi_{m}\eta_{n}}\right).
\end{eqnarray*}
In order to determine the behavior of $K_{\left|\alpha-\beta\right|}\left(2\sqrt{t\xi_{m}\eta_{n}}\right)$
near $t=0$, we first use \cite[13.6.10]{O-D-L-S:2016} to write 
\[
K_{\left|\nu\right|}(z)=\sqrt{\pi}(2z)^{\left|\nu\right|}e^{-z}U(\frac{1}{2}+\left|\nu\right|,1+2\left|\nu\right|,2z).
\]
The asymptotics of the function $U(\frac{1}{2}+\left|\nu\right|,1+2\left|\nu\right|,2z)$
near $z=0$ is fully described in \cite[13.5.6-12]{ABR-STE:1970}
for different parameters $\left|\nu\right|$. For $\alpha\ne\beta$
we obtain the asymptotics of $K_{\left|\alpha-\beta\right|}\left(\sqrt{t}\right)$
near $t=0$ as
\[
K_{\left|\alpha-\beta\right|}\left(\sqrt{t}\right)\sim t^{-\left|\alpha-\beta\right|/2}.
\]
Factoring out $t^{\left(\alpha+\beta-2\right)/2-\left|\alpha-\beta\right|/2}$,
we obtain \eqref{eq:final form} so that the function $v$ has at
most a logarithmic singularity at zero. Indeed, if $\alpha\ne\beta$
then $v$ has a finite value at $t=0$ whereas if $\alpha=\beta$
then the modified Bessel function $K_{0}$ has a logarithmic singularity
at $t=0$. 
\end{proof}
\begin{rem}
If in Lemma~\ref{lem:2} the exponents in the representations of
$f$ and $g$ (see \eqref{eq:form to maintain} and \eqref{eq:second PDF}),
satisfy $\mathcal{R}e\left(\xi_{m}\eta_{n}\right)>0$ for all $m,n$,
there is a simpler expression for the function $v$ in \eqref{eq:function_q}.
Using the integral

\begin{equation}
2^{p+1}\left(\frac{y}{x}\right){}^{p}K_{p}\left(xy\right)=\int_{-\infty}^{\infty}e^{-x^{2}e^{\tau}/4-y^{2}e^{-\tau}+p\tau}d\tau,\label{BesselK_integral}
\end{equation}
 (see \cite[Eq. 36]{BEY-MON:2010}) and setting $x=2\sqrt{\xi_{m}\eta_{n}}$,
$y=\sqrt{t}$, and $p=\left|\alpha-\beta\right|$, we have
\begin{equation}
2t^{\left|\alpha-\beta\right|/2}K_{\left|\alpha-\beta\right|}\left(2\sqrt{t\xi_{m}\eta_{n}}\right)=\left(\xi_{m}\eta_{n}\right)^{\left|\alpha-\beta\right|/2}\int_{-\infty}^{\infty}e^{-\xi_{m}\eta_{n}e^{\tau}-te^{-\tau}+\left|\alpha-\beta\right|\tau}d\tau.\label{BesselK_integral-2}
\end{equation}
Assuming $\alpha\ge\beta$, we obtain 
\begin{equation}
v\left(t\right)=\sum_{m=1}^{M}\sum_{n=1}^{N}a_{m}b_{n}\eta_{n}^{\left|\alpha-\beta\right|/2}\int_{-\infty}^{\infty}e^{-\xi_{m}\eta_{n}e^{\tau}-te^{-\tau}+\left|\alpha-\beta\right|\tau}d\tau\label{eq:function_q-1-1}
\end{equation}
or
\begin{equation}
v\left(t\right)=\int_{-\infty}^{\infty}\sigma\left(\tau\right)e^{-te^{-\tau}}d\tau,\label{eq:int-rep-simplified}
\end{equation}
where 
\begin{equation}
\sigma\left(\tau\right)=\sum_{m=1}^{M}\sum_{n=1}^{N}a_{m}b_{n}\eta_{n}^{\left|\alpha-\beta\right|}e^{-\xi_{m}\eta_{n}e^{\tau}+\left|\alpha-\beta\right|\tau}.\label{eq:weight-simplified}
\end{equation}
Assuming $\beta>\alpha$, we have
\[
\sigma\left(\tau\right)=\sum_{m=1}^{M}\sum_{n=1}^{N}a_{m}b_{n}\xi_{m}^{\left|\alpha-\beta\right|}e^{-\xi_{m}\eta_{n}e^{\tau}+\left|\alpha-\beta\right|\tau},
\]
and combining both cases, obtain
\[
\sigma\left(\tau\right)=\sum_{m=1}^{M}\sum_{n=1}^{N}a_{m}b_{n}\left(\text{sign}\left(\alpha-\beta\right)\frac{\eta_{n}-\xi_{m}}{2}+\frac{\eta_{n}+\xi_{m}}{2}\right)^{\left|\alpha-\beta\right|}e^{-\xi_{m}\eta_{n}e^{\tau}+\left|\alpha-\beta\right|\tau}.
\]
In order to represent the PDF $p$ of the product in the form \eqref{eq:form to maintain},
we need to approximate $v\left(t\right)$ as a linear combination
of exponentials. With that goal and following \cite{BEY-MON:2010},
we first discretize \eqref{eq:int-rep-simplified} using the trapezoidal
rule. Since the approximation obtained via this discretization may
have an excessive number of terms, we can then use the algorithm in
\cite{HAU-BEY:2012} to minimize their number.
\end{rem}

Finally, we show how to compute the PDF of the quotient $q$ of two
non-negative independent random variables $X$ and $Y$. 
\begin{lem}
\label{lem:3}The PDF $q$ in \eqref{eq:integral traditional form}
of the quotient $X/Y$ of two non-negative independent random variables
$X$ and $Y$, with PDFs $f$ and $g$ in \eqref{eq:form to maintain}
and \eqref{eq:second PDF}, can be written as 
\begin{equation}
q\left(t\right)=t^{\alpha-1}w\left(t\right),\label{eq:final form-1}
\end{equation}
where 
\begin{equation}
w\left(t\right)=\Gamma\left(\alpha+\beta\right)\sum_{m=1}^{M}\sum_{n=1}^{N}a_{m}b_{n}\left(\xi_{m}t+\eta_{n}\right)^{-\alpha-\beta}.\label{eq:function_q-1}
\end{equation}
\end{lem}

\begin{proof}
We rewrite \eqref{eq:integral to compute quotient} as
\begin{equation}
q\left(t\right)=t^{\alpha-1}\sum_{m=1}^{M}\sum_{n=1}^{N}a_{m}b_{n}\int_{0}^{\infty}y^{\alpha+\beta-1}e^{-\xi_{m}ty}e^{-\eta_{n}y}dy.\label{eq:to compute-quotient}
\end{equation}
Using
\[
\int_{0}^{\infty}y^{\nu-1}e^{-\gamma y}dy=\gamma^{-\nu}\Gamma\left(\nu\right),\,\,\,\mathcal{R}e\left(\nu\right)>0,\,\,\,\mathcal{R}e\left(\gamma\right)>0,
\]
we arrive at the result.
\end{proof}
Since we would like to maintain the form \eqref{eq:form to maintain}
of the PDFs of the sum $s$, product $p$ and quotient $q$, we seek
their representation as 
\[
s\left(t\right)=t^{\alpha+\beta-1}\sum_{k=1}^{J_{1}}c_{k}^{\left(1\right)}e^{-\omega_{k}^{\left(1\right)}t},\,\,\,\mathcal{R}e\left(\omega_{k}^{\left(1\right)}\right)>0,
\]
 
\[
p\left(t\right)=t^{\left(\alpha+\beta-2-\left|\alpha-\beta\right|\right)/2}\sum_{k=1}^{J_{2}}c_{k}^{\left(2\right)}e^{-\omega_{k}^{\left(2\right)}t},\,\,\,\mathcal{R}e\left(\omega_{k}^{\left(2\right)}\right)>0,
\]
and 
\[
q\left(t\right)=t^{\alpha-1}\sum_{k=1}^{J_{3}}c_{k}^{\left(3\right)}e^{-\omega_{k}^{\left(3\right)}t},\,\,\,\mathcal{R}e\left(\omega_{k}^{\left(3\right)}\right)>0
\]
where the number of terms $J_{1}$, $J_{2}$, and $J_{3}$ are (near)
optimal in each of these constructions. We solve this approximation
problem by first sampling at equally spaced nodes the functions $r$,
$v$ or $w$ of Lemmas~\ref{lem:1}-\ref{lem:3}, forming a Hankel
matrix with these samples, and then applying the algorithms described
in the next section. 

In what follows, we denote by $v$ the function that we seek to approximate
by a linear combination of decaying and possibly oscillatory exponentials,
\begin{equation}
\left|v\left(t\right)-\sum_{k=1}^{M}c_{k}e^{-\omega_{k}t}\right|\le\epsilon,\,\,\,\mathcal{R}e\left(\omega_{k}\right)>0,\label{eq:need to approximate}
\end{equation}
where the number of terms, $M$, is as small as possible.

\section{Algorithms for computing exponential representations\label{sec:Algorithm-for-computing}}

Numerical approximation of functions by exponentials can be understood
as a finite dimensional version of AAK theory \cite{AD-AR-KR:1968,AD-AR-KR:1968a,AD-AR-KR:1971};
this connection has been addressed in e.g. \cite{BEY-MON:2005,BEY-MON:2009,HAU-BEY:2012}.
Here we only briefly describe algorithms already developed for this
purpose. We show how to compute the exponents $\omega_{k}$ and coefficients
$c_{k}$ in \eqref{eq:need to approximate} from $2N+1$ equispaced
samples $v_{n}=v\left(Rn/\left(2N\right)\right)$, $n=0,\ldots,2N$,
where the range $R$ and the step size $R/\left(2N\right)$ are chosen
so that $v\left(t\right)$ is sufficiently over-sampled and $\left|v\left(t\right)\right|<\epsilon$,
$t\geq R$. Thus, we solve the discretized problem
\begin{equation}
\left|v\left(\frac{R}{2N}n\right)-\sum_{k=1}^{M}c_{k}\gamma_{k}^{n}\right|\le\epsilon,\,\,\,\left|\gamma_{k}\right|<1,\label{eq:discrete_approx}
\end{equation}
where we seek nodes $\gamma_{k}$ and coefficients $c_{k}$ so that
the number of terms $M$ is minimal. The exponents $\omega_{k}$ in
\eqref{eq:need to approximate} are related to the nodes $\gamma_{k}$
by 
\[
\gamma_{k}=e^{-\frac{R}{2N}\omega_{k}}.
\]
Currently there are two algorithms for obtaining the approximation
\eqref{eq:discrete_approx}; both use the Hankel matrix constructed
from the samples $v_{n}$, $H=\left[v_{i+j}\right]_{i,j=0}^{N}$(see
\cite{KU-AR-BH:1983,HUA-SAR:1988,HUA-SAR:1990,HUA-SAR:1991} and \cite{BEY-MON:2005,RE-BE-MO:2013a}). 

We first present the key steps of the so-called HSVD (or matrix pencil)
algorithm \cite{KU-AR-BH:1983,HUA-SAR:1988,HUA-SAR:1990,HUA-SAR:1991}.
In Algorithm~\ref{alg:Computing-exponential-representations-I},
$X^{\dagger}$ denotes the pseudo-inverse of the matrix $X$, $X\left(m:n,:\right)$
denotes the sub-matrix consisting of rows $m$ through $n$, $X=\left(\mathbf{x}_{1}\dots\mathbf{x}_{M}\right)$
denotes the matrix consisting of the column vectors $\mathbf{x}_{1}$,
$\mathbf{x}_{2}$, ... $\mathbf{x}_{M}$ and $\mathbf{v}$ denotes
the vector of samples, $\mathbf{v}=\left(v_{0},\dots,v_{2N}\right)$.

\begin{algorithm}[h]
\caption{Computing exponential representations I\foreignlanguage{english}{\label{alg:Computing-exponential-representations-I}}}
\begin{enumerate}
\item For a desired accuracy $\epsilon$, compute $M$ con-eigenvectors
and corresponding con-eigenvalues of $H=\left[v_{i+j}\right]_{i,j=0}^{N},$
$H\mathbf{u}_{m}=\sigma_{m}\overline{\mathbf{u}}_{m}$, $m=0,\ldots,M-1$,
where $\sigma_{0}\ge\sigma_{1}\ge\dots\ge\sigma_{M-1}\ge\sigma_{M}$,
such that $\sigma_{M}/\sigma_{0}<\epsilon$. A solution to this problem
is guaranteed by Tagaki's factorization \cite{HOR-JOH:1990} and may
be reduced to finding the SVD of $H$.
\item Form the $M\times M$ matrix $U_{3}=U_{1}^{\dagger}U_{2}$ , where
$U=\left(\mathbf{u_{0}}\dots\mathbf{u}_{M-1}\right)$, $U_{1}=U\left(0:N-1,1:M\right)$,
and $U_{2}=U\left(1:N,1:M\right)$.
\item Compute the $M$ eigenvalues of $U_{3}$. They coincide with the nodes
$\gamma_{k}$ in \eqref{eq:discrete_approx}.
\item Compute the coefficients $c_{k}$ in \eqref{eq:discrete_approx} via
$\mathbf{c}=V^{\dagger}\mathbf{v}$, where $V$ is the $\left(2N+1\right)\times M$
Vandermonde matrix of entries $V_{nk}=\gamma_{k}^{n}$, $k=1,\ldots,M$
and $n=0,\ldots,2N$.
\end{enumerate}
\end{algorithm}

An alternative algorithm, described below as Algorithm~\ref{alg:Computing-exponential-representation-II},
was introduced in \cite{BEY-MON:2005} (see also \cite{BEY-MON:2010,HAU-BEY:2012,RE-BE-MO:2013a});
it relies on solving a con-eigenvalue problem for the Hankel matrix
$H$ (solution of which is guaranteed by Tagaki's factorization \cite{HOR-JOH:1990})
and may be reduced to finding the Singular Value Decomposition (SVD)
of $H$. Unlike Algorithm~\ref{alg:Computing-exponential-representations-I},
it requires a single con-eigenvector of the same Hankel matrix as
in Algorithm~\ref{alg:Computing-exponential-representations-I}.
We have implemented and used both algorithms. 

\begin{algorithm}
\caption{Computing exponential representations II\label{alg:Computing-exponential-representation-II}}
\begin{enumerate}
\item Given $\epsilon$, the desired accuracy, compute the con-eigenvector
$\mathbf{u}_{M}$ and the corresponding con-eigenvalue $\sigma_{M}$
of $H=\left[v_{i+j}\right]_{i,j=0}^{N}$, $H\mathbf{u}_{M}=\sigma_{M}\overline{\mathbf{u}}_{M}$,
such that $\sigma_{M}/\sigma_{0}<\epsilon$, where $\sigma_{0}$ is
the largest con-eigenvalue. A solution is guaranteed by Tagaki's factorization
\cite{HOR-JOH:1990} and may be reduced to finding the $M+1$ singular
vector of $H$.
\item Compute the roots $\gamma_{j}$ of the polynomial 
\begin{equation}
u(z)=\sum_{l=0}^{N}u_{l}z^{l},\label{eq: Eq. for roots}
\end{equation}
where $\mathbf{u}_{M}=\left\{ u_{l}\right\} _{l=0}^{N}$.
\item The exponents $\omega_{k}$ in equation \eqref{eq:need to approximate}
are defined by the roots $\gamma_{k}$ inside the unit disk via $\omega_{k}=-\frac{2N}{R}\log(\gamma_{k})$,
where $\log$ is the principal value of the logarithm.
\item Compute the coefficients $c_{k}$ in \eqref{eq:discrete_approx} via
$\mathbf{c}=V^{\dagger}\mathbf{v}$, where $V$ is the $\left(2N+1\right)\times M$
Vandermonde matrix of entries $V_{nk}=\gamma_{k}^{n}$, $k=1,\ldots,M$
and $n=0,\ldots,2N$.
\end{enumerate}
\end{algorithm}

Importantly, we implemented a fast SVD solver for Step~$1$ of Algorithms~\ref{alg:Computing-exponential-representations-I}~and~\ref{alg:Computing-exponential-representation-II}
using the randomized approach developed in \cite{C-G-M-R:2005,L-W-M-R-T:2007,HA-MA-TR:2011}
and the fact that Hankel matrices can be applied in $\mathcal{O}\left(N\log\left(N\right)\right)$
operations using the Fast Fourier Transform (FFT). We describe it
as Algorithm~\ref{alg:Computing-fast-SVD-Hankel} and note that the
number $M$ of singular vectors needed to achieve accuracy $\epsilon$
is usually unknown. If $M$ is chosen correctly, then the smallest
pivots computed in Step~2 of Algorithm~\ref{alg:Computing-fast-SVD-Hankel}
will be less than $\epsilon$. However, if all pivots are greater
than $\epsilon$, then by doubling $M$, Steps~1~and~2 of Algorithm~\ref{alg:Computing-fast-SVD-Hankel}
can be repeated until the desired size of pivots is achieved. 

\begin{algorithm}[h]
\caption{Computing fast SVD of a low rank Hankel matrix \label{alg:Computing-fast-SVD-Hankel}}
\begin{enumerate}
\item Apply the Hankel matrix $H$ to $M'$ random normally distributed
vectors $Y$, where $M'\ge M+p$, to obtain the $N+1\times M'$ matrix
$L=HY$. Here $p$ is the so-called oversampling parameter, essentially
a constant, see \cite{HA-MA-TR:2011} for details. This step requires
$\mathcal{O}\left(MN\log N\right)$ operations.
\item Compute the rank revealing (pivoted) QR decomposition of the matrix
$L=QR$, where $Q$ is $N+1\times M'$. This step requires $\mathcal{O}\left(M^{2}N\right)$
operations.
\item Apply the adjoint Hankel matrix $H^{*}$ to $Q$ to obtain $H^{*}Q$,
an $N+1\times M'$ matrix. This step requires $\mathcal{O}\left(MN\log N\right)$
operations.
\item Compute the SVD of $H^{*}Q=U\Sigma V^{*}$, so that $H=QV\Sigma U^{*}$.
Since the Hankel matrix $H$ has a Tagaki's decomposition, $H=\overline{U}\Sigma U^{*}$,
the computed matrix $U$ may differ from the one in the Tagaki's factorization
by a factor in the form of a diagonal matrix with diagonal elements
of modulus one. This unknown factor does not play a role in Step~2
of Algorithm~\ref{alg:Computing-exponential-representations-I} as
it cancels out. The $N+1\times M'$ matrix $\overline{U}$ is the
desired matrix for Step~2 of Algorithm~\ref{alg:Computing-exponential-representations-I}.
This step requires $\mathcal{O}\left(M^{2}N\right)$ operations.
\end{enumerate}
\end{algorithm}

The Fast SVD construction in Algorithm~\ref{alg:Computing-fast-SVD-Hankel}
reduces the overall cost of Algorithms~\ref{alg:Computing-exponential-representations-I}~and~\ref{alg:Computing-exponential-representation-II}
to $\mathcal{O}\left(MN\log\left(N\right)+M^{2}N\right)$ operations,
where the (implicit) constant is small. Since in our application $M\ll N$
(e.g. $M=20$), the cost of the algorithm is essentially linear in
the number of samples $N$. In our experience, the choice of the singular
value $\sigma_{M}$, $\sigma_{M}/\sigma_{0}<\epsilon$, in Algorithm~\ref{alg:Computing-exponential-representations-I}~and~\ref{alg:Computing-exponential-representation-II}
always results in an $\mathcal{O}\left(\epsilon\right)$ error bound.

We note that we can always check the approximation error a posteriori,
using e.g. the already computed values $v_{n}=v\left(Rn/\left(2N\right)\right)$,
and select a smaller singular value, if necessary. The connection
between the accuracy $\epsilon$ and the ratio of the $M$th and the
largest singular values, $\sigma_{M}/\sigma_{0}$ in Algorithms~\ref{alg:Computing-exponential-representations-I}~and~\ref{alg:Computing-exponential-representation-II}
is one of the key features of AAK theory \cite{AD-AR-KR:1968,AD-AR-KR:1968a,AD-AR-KR:1971}
for semi-infinite Hankel matrices.
\begin{rem}
\label{rem: repeated approxiamtion}The function $v$ in \eqref{eq:need to approximate}
may change rapidly near zero (e.g. it can have a logarithmic singularity
at zero) so that it requires sampling with a small step size. Due
to the equally spaced sampling of $v$ (see \eqref{eq:discrete_approx}),
the size of the matrix $H$ in Algorithms~\ref{alg:Computing-exponential-representations-I}
and \ref{alg:Computing-exponential-representation-II} can be large.
Although we have a fast algorithm for computing the SVD of a large
matrix (Algorithm~\ref{alg:Computing-fast-SVD-Hankel}), we can instead
apply Algorithm~\ref{alg:Computing-exponential-representations-I}
or \ref{alg:Computing-exponential-representation-II} several times
using first a coarse sampling (sufficient in an interval away from
zero) and then subtracting the result from $v$ so that the essential
support of the difference is reduced. Specifically, given a function
$v\left(t\right)$ with $\left|v\left(t\right)\right|<\epsilon$,
$t\geq R$, we approximate it by $\tilde{v}\left(t\right)$,
\begin{equation}
\tilde{v}\left(t\right)=\sum_{i=1}^{K}\tilde{v}_{i}\left(t\right),\label{eq:combined approximation}
\end{equation}
where each $\tilde{v}_{i}\left(t\right)$ is obtained using Algorithms~\ref{alg:Computing-exponential-representations-I}
or \ref{alg:Computing-exponential-representation-II} by sampling
the function $v_{i}$
\end{rem}

\[
v_{i}\left(t\right)=v\left(t\right)-\sum_{j=1}^{i-1}\tilde{v}_{j}\left(t\right),\ \ \ i=1,\ldots K
\]
on the interval $\left[0,b_{i}\right]$, with $b_{1}=R,$ and $b_{k}<b_{k-1}$,
$i=2,\ldots,K$. 

For instance, in Example~\ref{subsec:gam205305} this procedure was
applied on the intervals $\left[0,100\right]$, $\left[0,1\right]$
and $\left[0,10^{-2}\right]$ using $N=2000$ each time. The result
was accurate within the selected $\epsilon$ the first time on the
interval $\left[1,100\right]$, the second time on $\left[10^{-2},1\right]$,
and finally on $\left[10^{-4},10^{-2}\right]$. The resulting approximation
\eqref{eq:combined approximation} is accurate on the interval $\left[10^{-4},100\right]$
(see Figure~\ref{fig:gam205305-1}).

\section{Examples of computing PDFs of products of random variables\label{sec:Examples-of-computing}}

\subsection{Accuracy tests}

In a few cases, the PDFs of the product of positive random variables
are available analytically; we use those cases to demonstrate the
accuracy of our algorithm. 

\subsubsection{\label{subsec:gam205305}Product of two Gamma random variables}

The PDF of the Gamma distribution is given by 

\begin{equation}
f^{\gamma}\left(x;\alpha,\beta\right)=\frac{\beta^{\alpha}}{\Gamma\left(\alpha\right)}x^{\alpha-1}e^{-\beta x},\label{eq:gamma-pdf}
\end{equation}
where $\alpha>0$ and $\beta>0$ are called the shape and rate parameters.
Worth noting are two special cases of the Gamma distribution: when
$\alpha=1$ it is called exponential distribution, when $\beta=1/2$
it is known as chi-squared distribution. We also note that the PDF
of the Gamma distribution is already in the form (\ref{eq:form to maintain})
that we want to maintain. For Gamma-distributed random variables $X\sim f^{\gamma}(x;2,2)$
and $Y\sim f^{\gamma}(x;3,2)$ we compute the PDF $p_{Z}$ of their
product, $Z=XY$. The product PDF is available analytically as $p\left(t\right)=32t^{3/2}K_{1}\left(4\sqrt{t}\right)$,
where $K_{1}$ is a modified Bessel function of the second kind. Using
(\ref{eq:final form}) we compute the PDF of the product $p_{Z}$
and compare it with the analytic result $p$. The PDFs of the random
variables $X$ and $Y$ are displayed in Figure~\ref{fig:gam205305}
and the error, defined as
\begin{equation}
\epsilon\left(x\right)=\log_{10}\left(\left|p_{Z}\left(10^{x}\right)-p(10^{x})\right|\right),\,\,\,-5\le x\le1,\label{eq:gamma-err}
\end{equation}
is shown in Figure~\ref{fig:gam205305-1}.

\begin{figure}[H]
\begin{centering}
\includegraphics[scale=0.3]{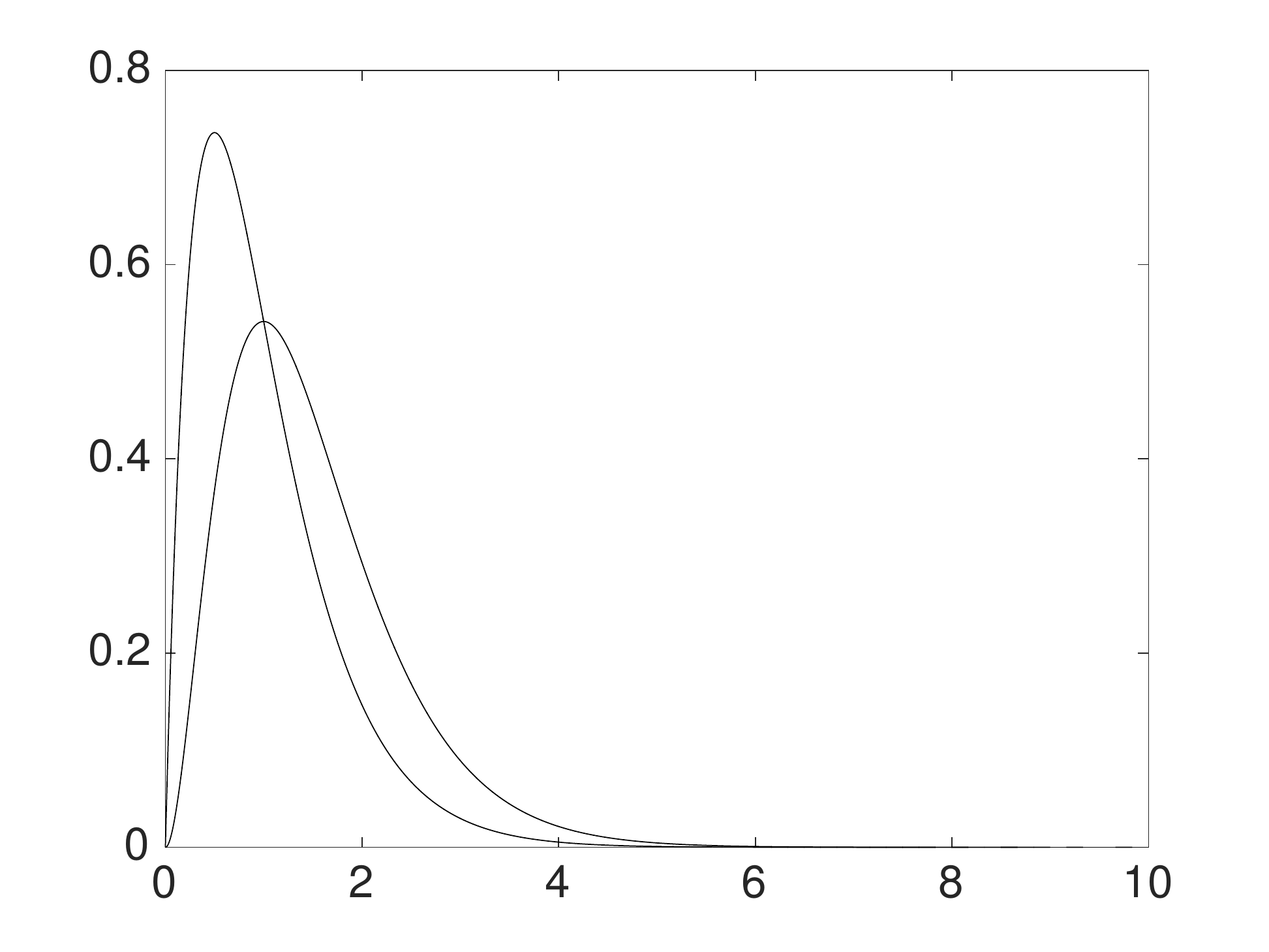}\includegraphics[scale=0.3]{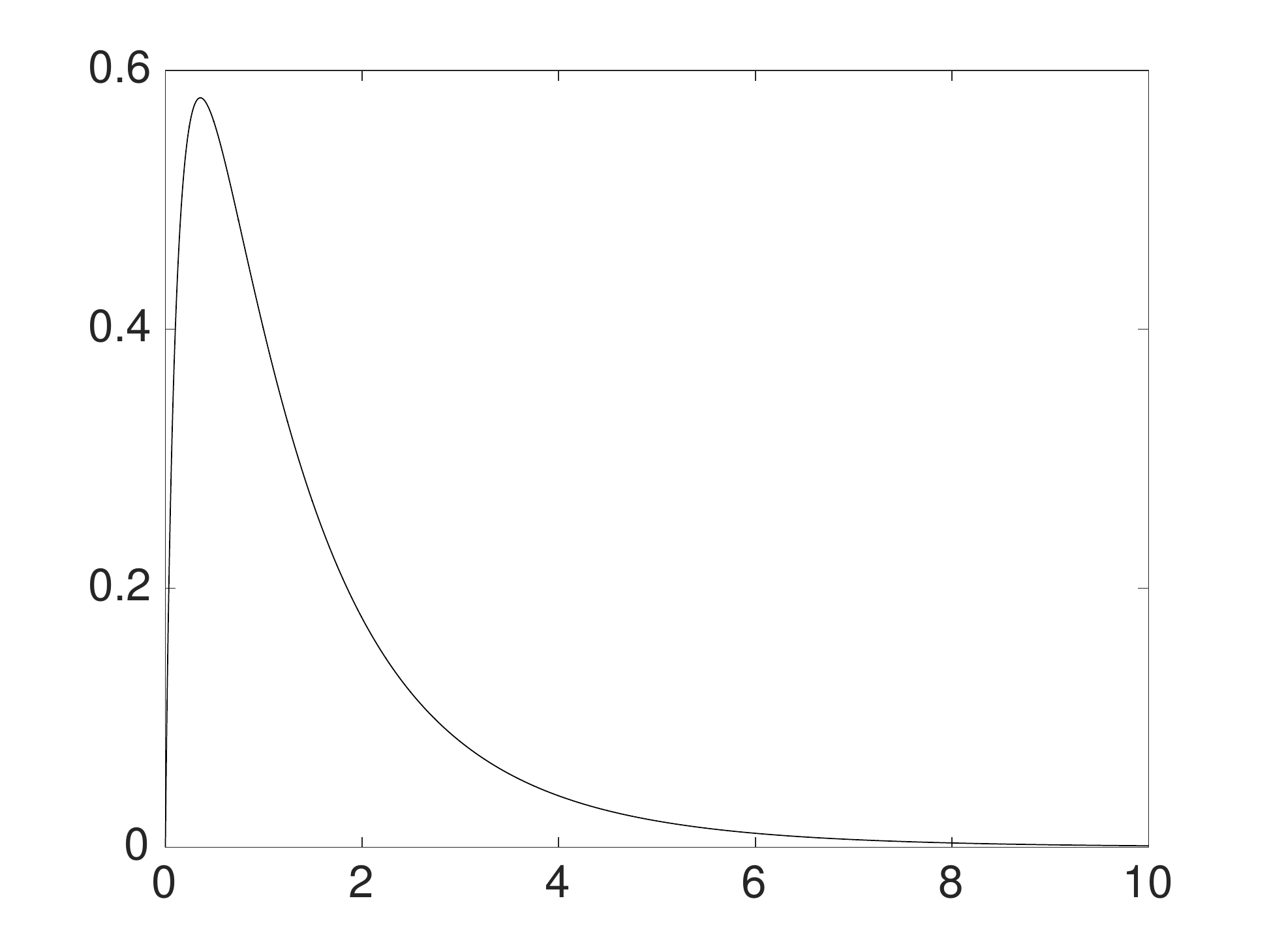}
\par\end{centering}
\centering{}\caption{\label{fig:gam205305} PDFs of the random variables $X$ and $Y$
(left) and PDF of the product random variable $Z=XY$ in Example~\ref{subsec:gam205305}
(right).}
\end{figure}
\begin{figure}[H]
\begin{centering}
\includegraphics[scale=0.3]{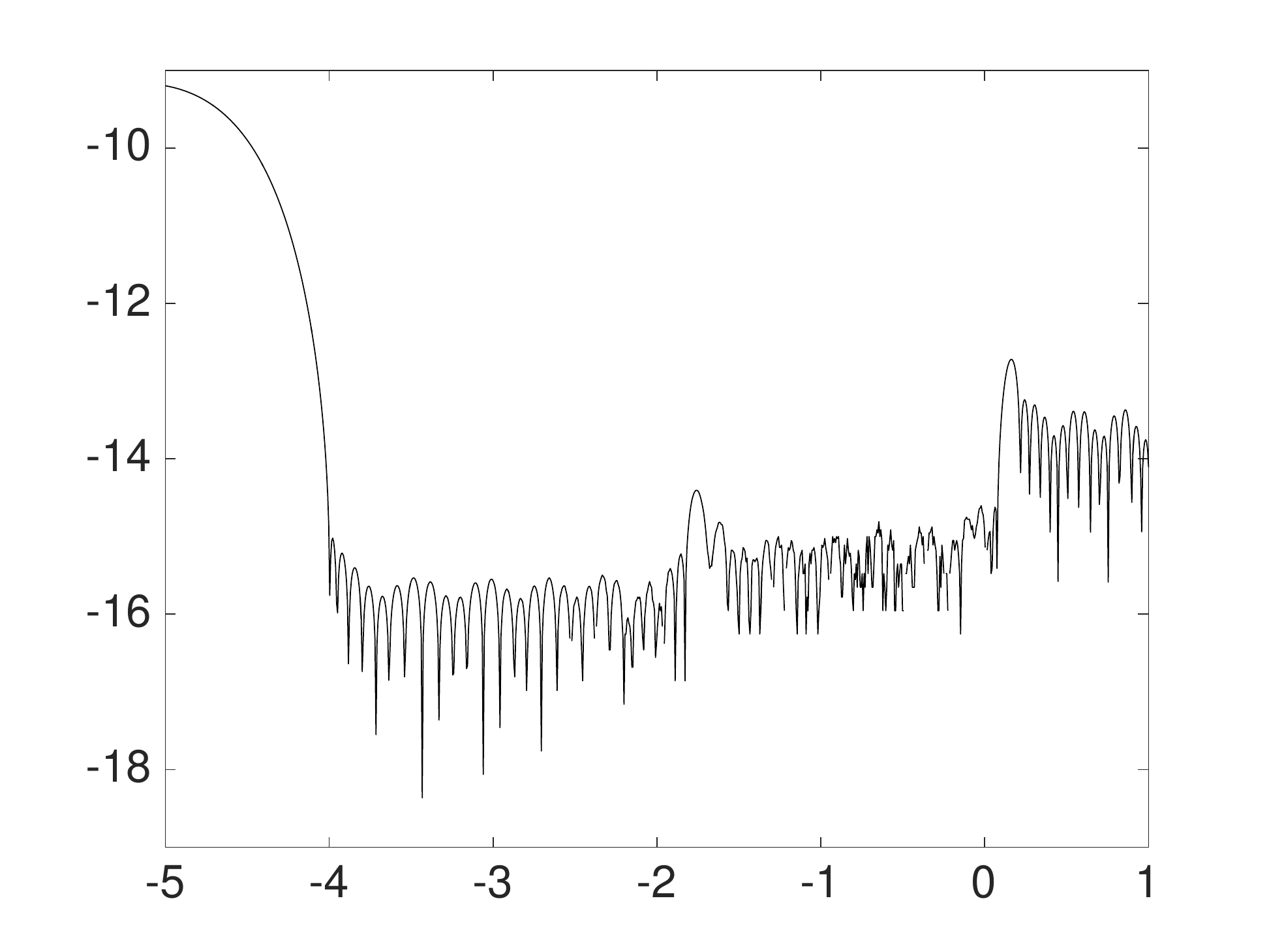}
\par\end{centering}
\centering{}\caption{\label{fig:gam205305-1} Error curve $\epsilon\left(x\right)$ (see
(\ref{eq:gamma-err})) in Example~\ref{subsec:gam205305}. Algorithm~\ref{alg:Computing-exponential-representation-II}
was applied three times to obtain this result (see Remark~\ref{rem: repeated approxiamtion}).}
\end{figure}

\subsubsection{\label{subsec:nak11}Product of Nakagami random variables}

The distributions of the product of Nakagami random variables have
applications in wireless communication systems \cite{Z-W-H-T:2012}.
The PDF of a Nakagami distributed random variable is given by

\begin{equation}
f^{\mathcal{N}}\left(x;m,\Omega\right)=\frac{2m^{m}}{\Gamma\left(m\right)\Omega^{m}}x^{2m-1}e^{-\frac{m}{\Omega}x^{2}},\label{eq:nak-pdf}
\end{equation}
where $m\geq1/2$ and $\Omega>0$ are called the shape and spread
parameters; as is well known, the Nakagami distribution is related
to the Gamma and Chi distributions. In this example, we compute the
PDFs of the product of two, four and eight Nakagami distributed random
variables. Given the random variable $X\sim f^{\mathcal{N}}\left(x,1,1\right)=2xe^{-x^{2}}$
(see Figure~\ref{fig:nak11x2}), we first employ either Algorithm~\ref{alg:Computing-exponential-representations-I}
or \ref{alg:Computing-exponential-representation-II} on the Gaussian
part of $f^{\mathcal{N}}$, $g\left(x\right)=2e^{-x^{2}}$, to obtain
its approximation, $\tilde{g}\left(x\right)$, in the form (\ref{eq:form to maintain})
with $\alpha=1$. To obtain $\tilde{g}$ it is sufficient to sample
$g$ on the interval $x\in\left[0,6\right]$ and use Algorithm~\ref{alg:Computing-exponential-representation-II}
to solve (\ref{eq:discrete_approx}) with $R=6$ and $N=500$, where
we set $\epsilon=10^{-11}$. In Figure~\ref{fig:gauss-err} we display
the error, 

\begin{equation}
\epsilon\left(x\right)=\log_{10}\left(\left|g\left(10^{x}\right)-\tilde{g}(10^{x})\right|\right),\,\,\,-12\le x\le1,\label{eq:gauss-err}
\end{equation}
of the approximation via $\tilde{g}(x)$. Figure~\ref{fig:gauss-nodes}
shows the location of the complex nodes $\xi_{m}$ in the resulting
representation of $\tilde{g}(x)$.

\begin{figure}[H]
\begin{centering}
\includegraphics[scale=0.3]{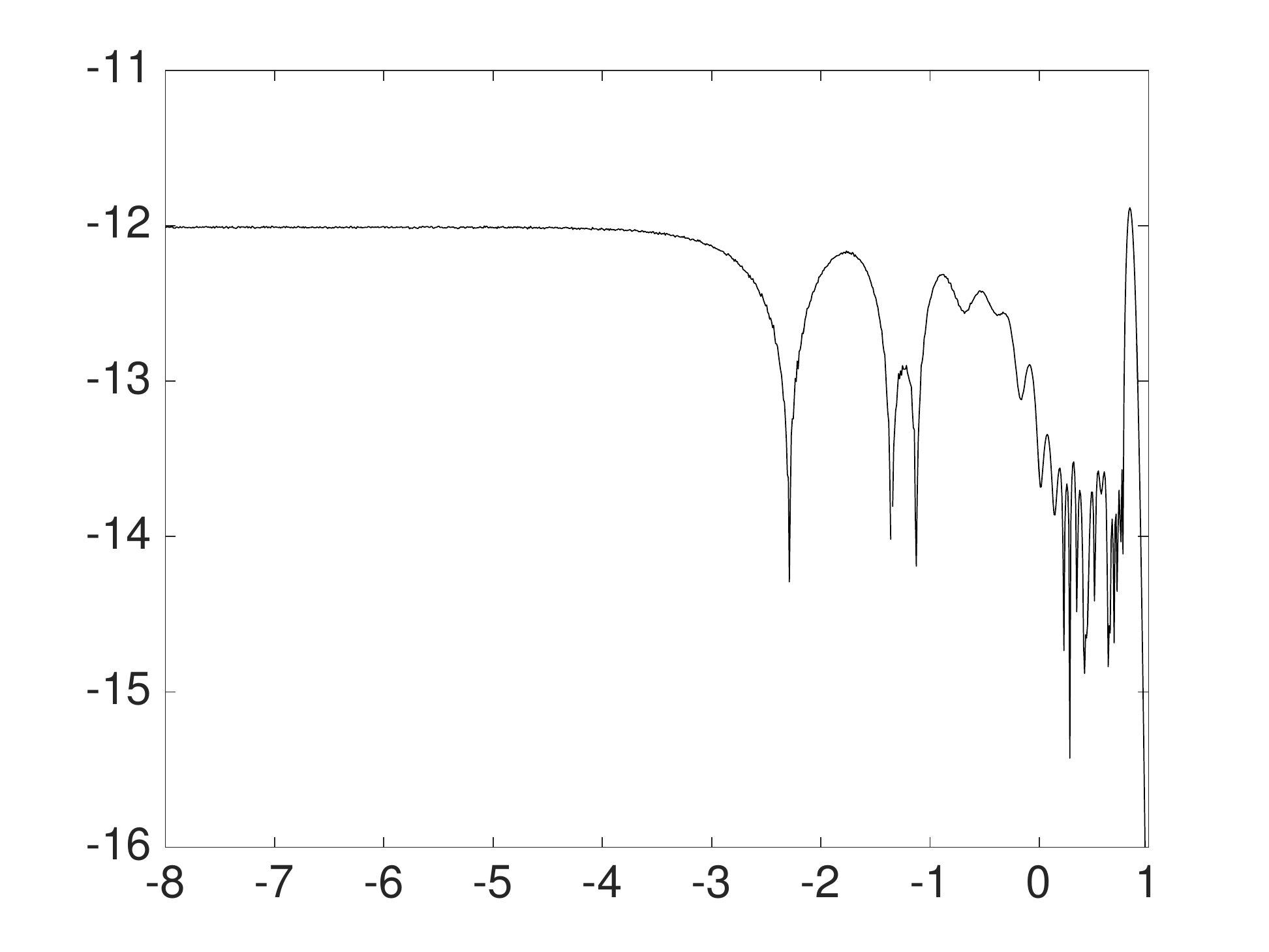}
\par\end{centering}
\centering{}\caption{\label{fig:gauss-err} Error curve $\epsilon\left(x\right)$ (see
(\ref{eq:gauss-err})) in Example~\ref{subsec:nak11}.}
\end{figure}

\begin{figure}[H]
\begin{centering}
\includegraphics[scale=0.3]{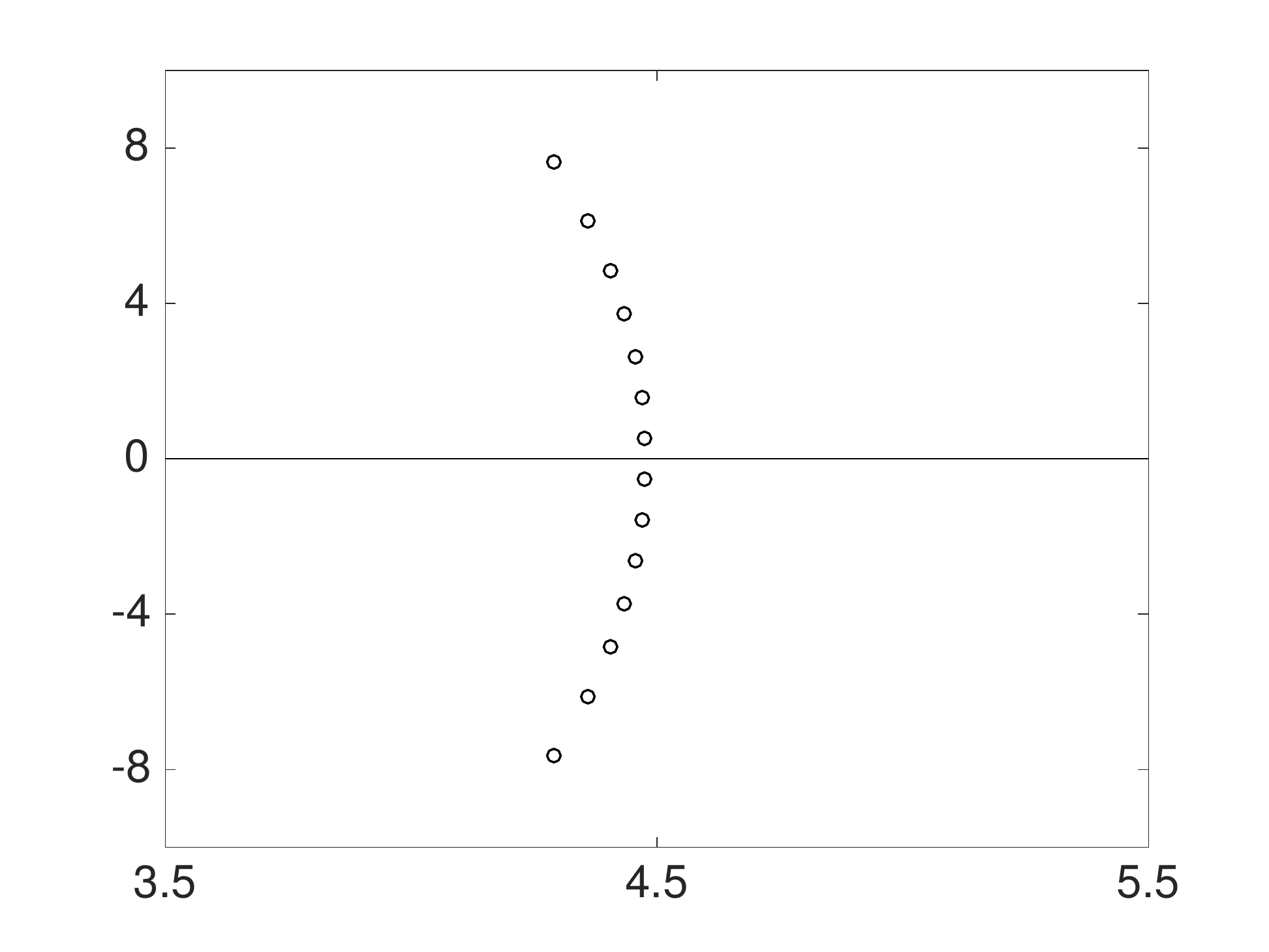}
\par\end{centering}
\centering{}\caption{\label{fig:gauss-nodes} Complex nodes $\xi_{m}$ in the representation
of $\tilde{g}(x)$ in Example~\ref{subsec:nak11}. Horizontal and
vertical axes correspond to $\mathcal{R}e\left(\xi_{m}\right)$ and
$\mathcal{I}m\left(\xi_{m}\right)$, respectively.}
\end{figure}

After obtaining an accurate approximation $f_{X}$ of $f^{\mathcal{N}}$,
using $\tilde{g}$ in the form (\ref{eq:form to maintain}), we compute
the PDF $p_{Y}$ of the random variable $Y=X^{2}$ using (\ref{eq:to compute})
and display it in Figure~\ref{fig:nak11x2}. The exact product PDF
is available analytically as $p\left(x\right)=4xK_{0}\left(2x\right)$,
where $K_{0}$ is modified Bessel function of the second kind. The
error,

\begin{equation}
\epsilon\left(x\right)=\log_{10}\left(\left|p_{Y}\left(10^{x}\right)-p(10^{x})\right|\right),\,\,\,-8\le x\le1,\label{eq:nak11x2-err}
\end{equation}
is displayed in Figure~\ref{fig:nak11x2-err}. Figure~\ref{fig:nak11x2-nodes}
shows the location of the complex nodes $\xi_{m}$ in the representation
of $p_{Y}$.

Using the computed PDF $p_{Y}$, we compute the PDF $p_{Z}$ of the
product of four Nakagami distributed random variables, $Z=Y^{2}=X^{4}$
(see Figure~\ref{fig:nak11x4}). Likewise, we compute the PDF $p_{W}$
of the product of eight Nakagami random variables and display the
result in Figure~\ref{fig:nak11x8}.

\begin{figure}[H]
\begin{centering}
\includegraphics[scale=0.3]{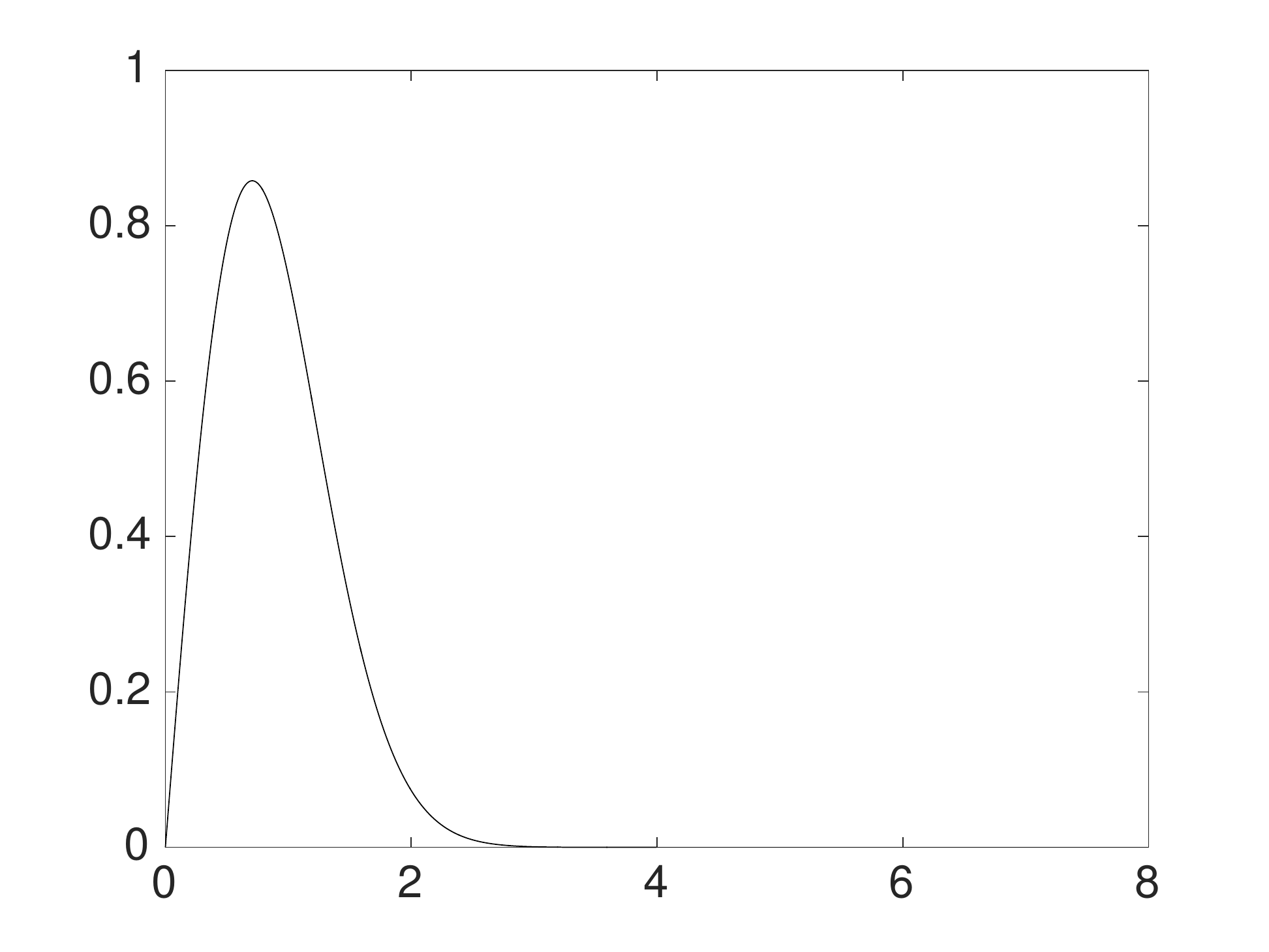}\includegraphics[scale=0.3]{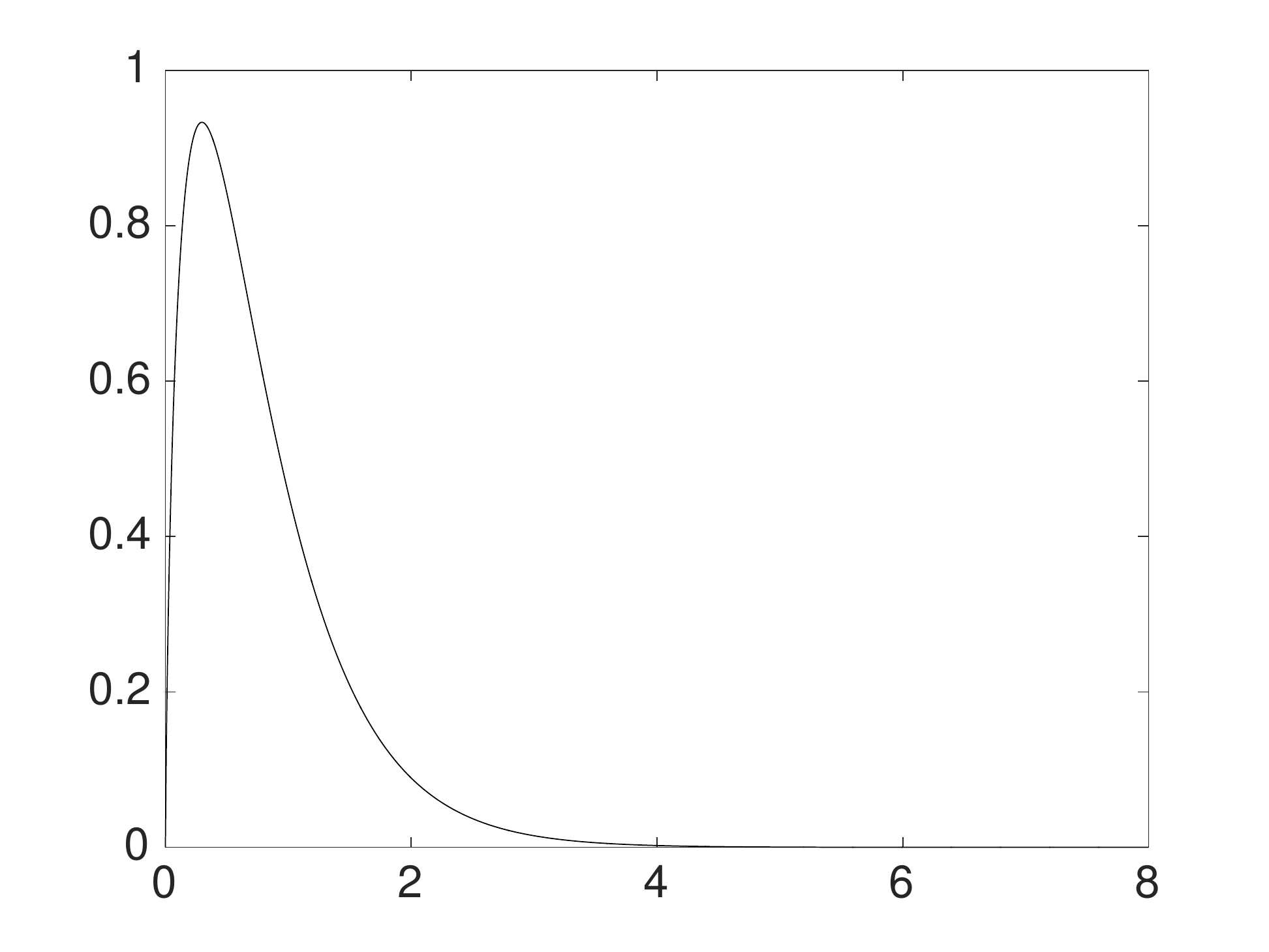}
\par\end{centering}
\centering{}\caption{\label{fig:nak11x2} PDF of the random variable $X$ (left) and PDF
of the product random variable $Y=X^{2}$ in Example~\ref{subsec:nak11}
(right).}
\end{figure}

\begin{figure}[H]
\begin{centering}
\includegraphics[scale=0.3]{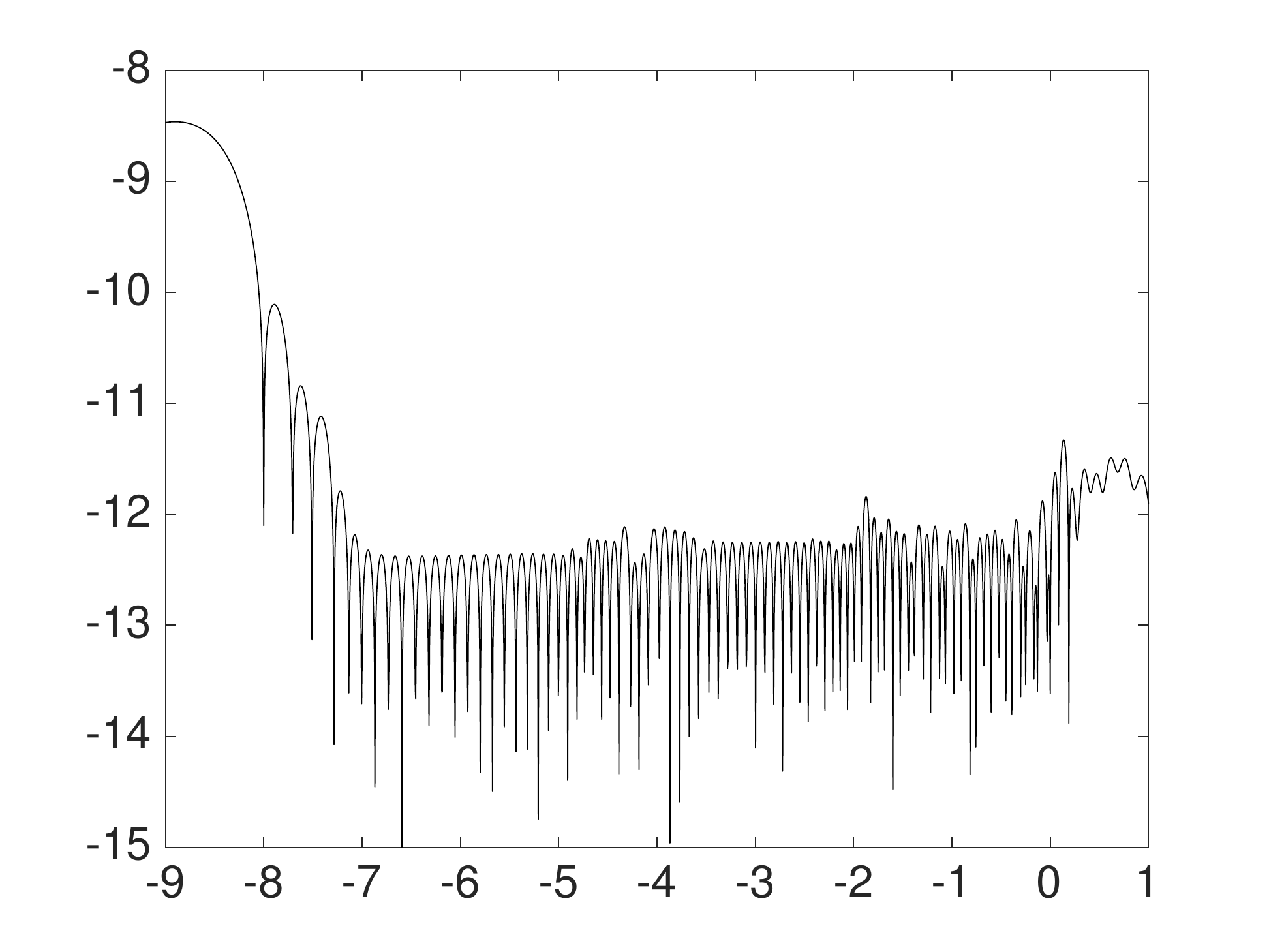}
\par\end{centering}
\centering{}\caption{\label{fig:nak11x2-err} Error curve $\epsilon\left(x\right)$ (see
(\ref{eq:nak11x2-err})) in Example~\ref{subsec:nak11}. Algorithm~\ref{alg:Computing-exponential-representation-II}
was applied four times to obtain this result (see Remark~\ref{rem: repeated approxiamtion}).}
\end{figure}

\begin{figure}[H]
\begin{centering}
\includegraphics[scale=0.3]{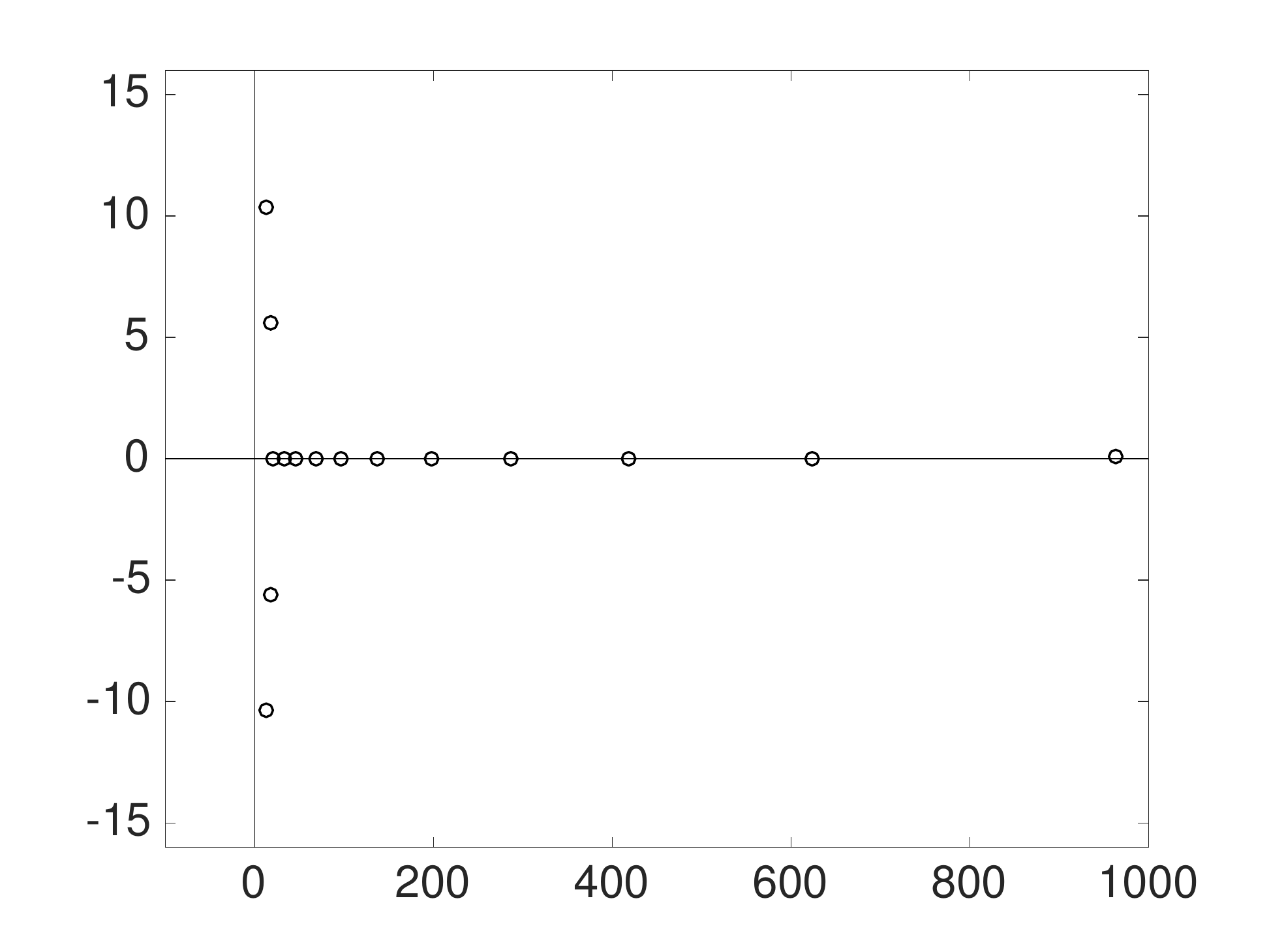}
\par\end{centering}
\centering{}\caption{\label{fig:nak11x2-nodes} Complex nodes $\xi_{m}$ in the representation
of the PDF $p_{Y}$ in Example~\ref{subsec:nak11}. Horizontal and
vertical axes correspond to $\mathcal{R}e\left(\xi_{m}\right)$ and
$\mathcal{I}m\left(\xi_{m}\right)$, respectively.}
\end{figure}

\begin{figure}[H]
\begin{centering}
\includegraphics[scale=0.3]{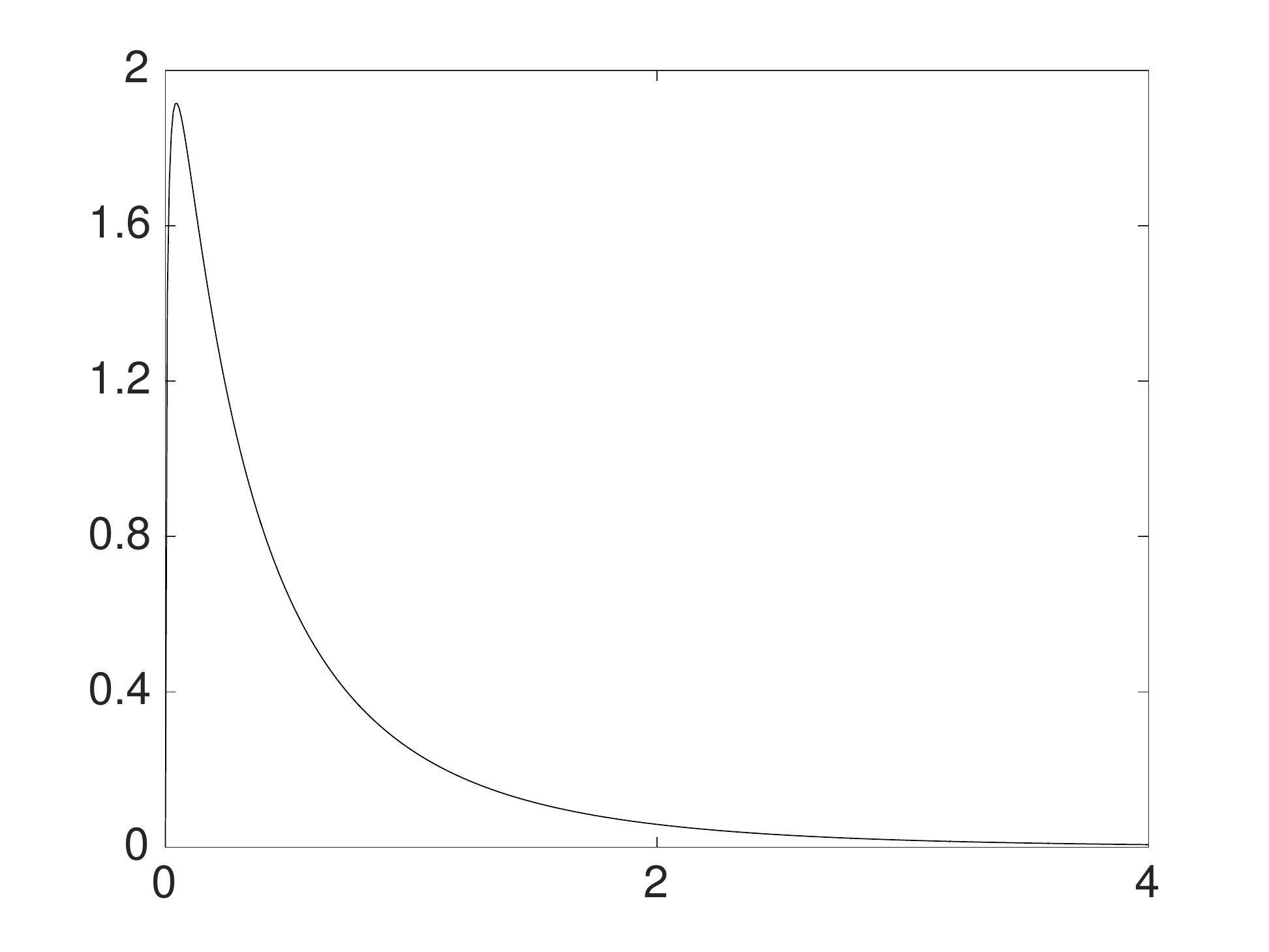}
\par\end{centering}
\centering{}\caption{\label{fig:nak11x4} PDF of the random variable $Z=X^{4}$ in Example~\ref{subsec:nak11}.}
\end{figure}

\begin{figure}[H]
\begin{centering}
\includegraphics[scale=0.3]{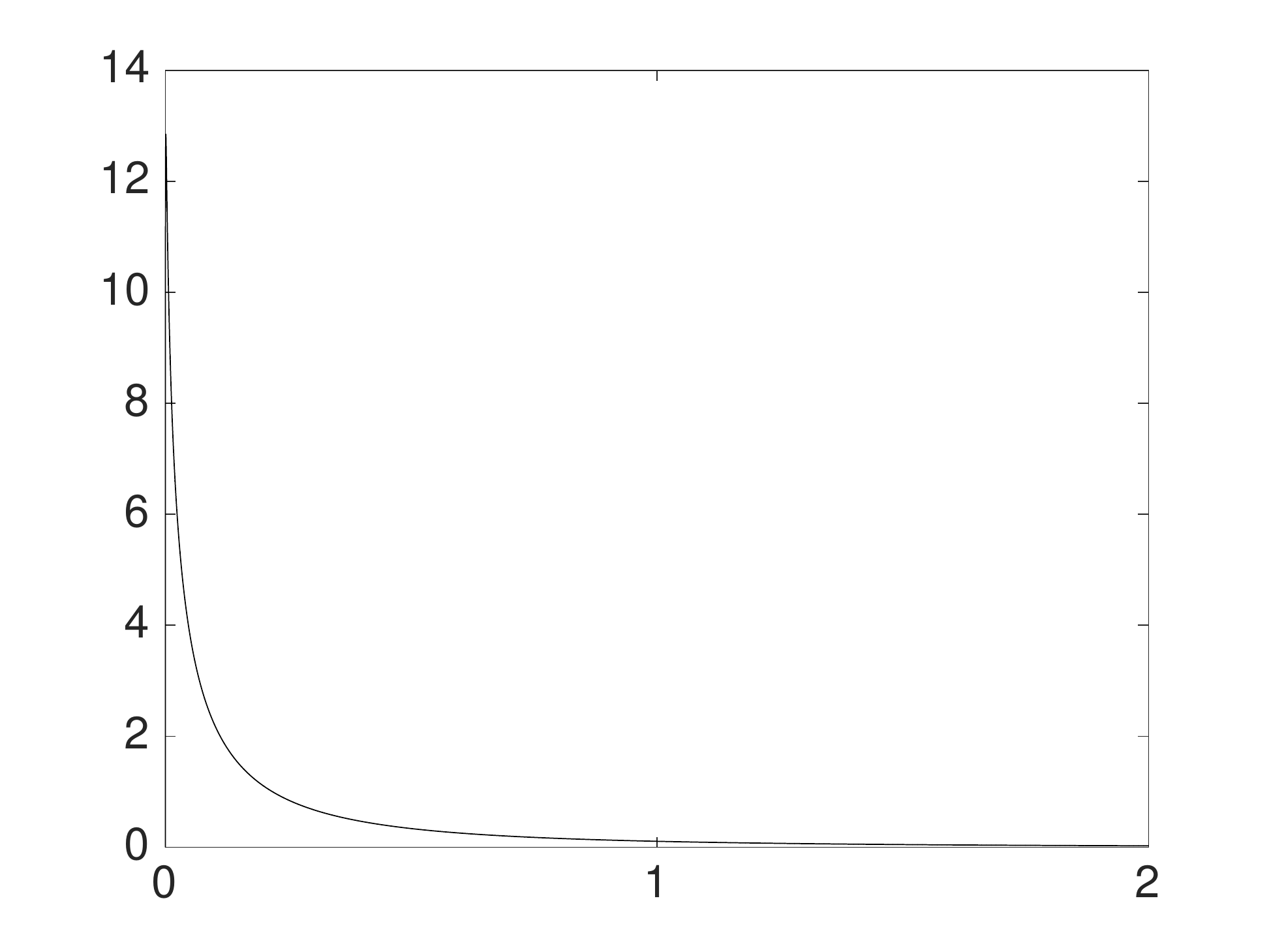}\includegraphics[scale=0.3]{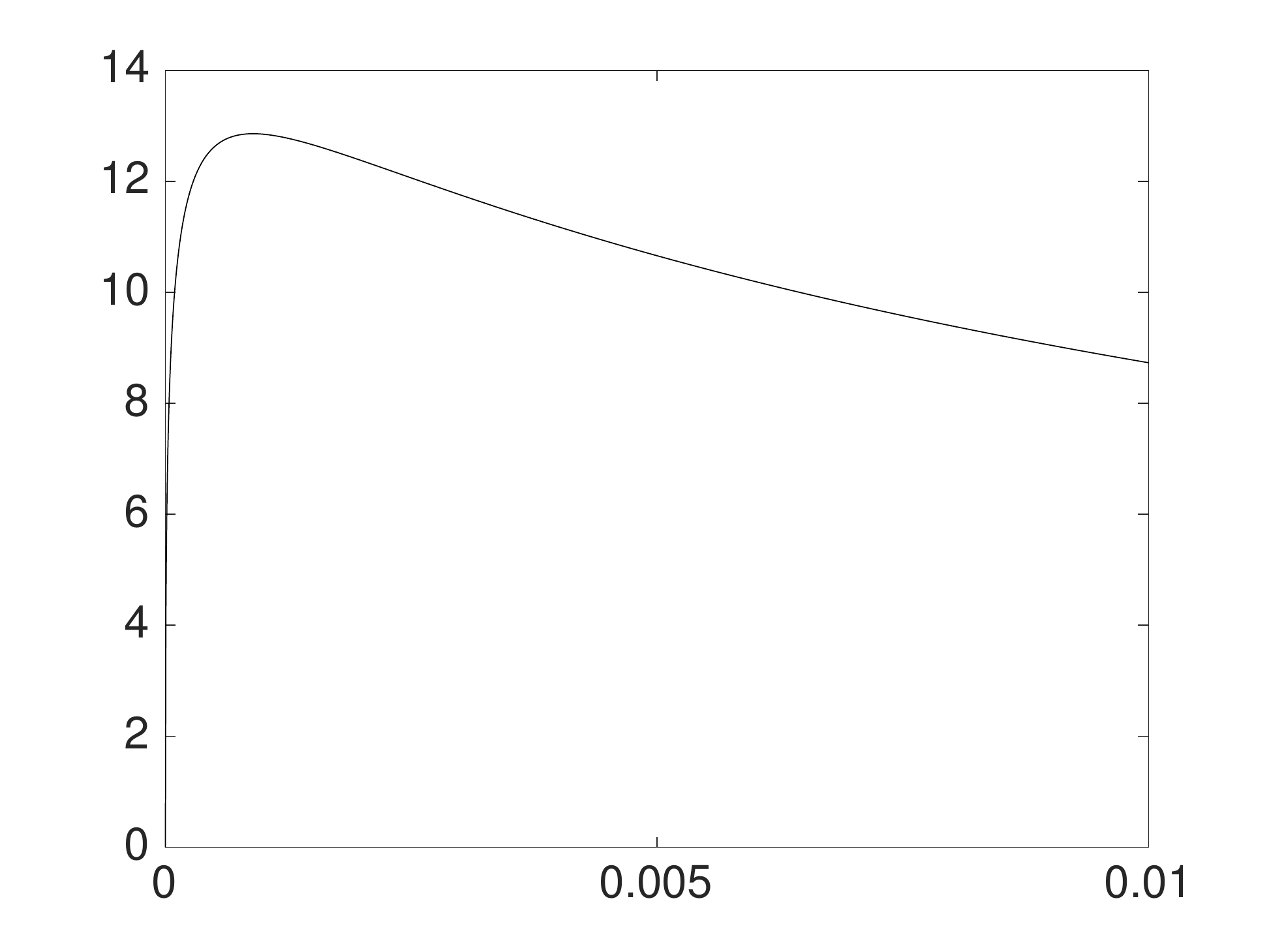}
\par\end{centering}
\centering{}\caption{\label{fig:nak11x8} PDF of the random variable $W=X^{8}$ on the
interval $x\in\left[0,2\right]$ (left) and on the interval $x\in\left[0,0.1\right]$
in Example~\ref{subsec:nak11} (right).}
\end{figure}

\subsubsection{\label{subsec:lom-gam}Product of Lomax and Gamma random variables}

As another example, we compute the PDF of the product of a Gamma distributed
random variable with PDF given by (\ref{eq:gamma-pdf}) and a Lomax
distributed random variable with PDF given by

\[
f^{\ell}\left(x;\alpha,\lambda\right)=\frac{\alpha}{\lambda}\left(1+\frac{x}{\lambda}\right)^{-\left(\alpha+1\right)},
\]
where $\alpha>0$ and $\lambda>0$ are called the shape and scale
parameters. In this example we use the Gamma distributed random variable
$X\sim f^{\gamma}\left(x;3,2\right)$ and the Lomax distributed random
variable $Y\sim f^{\ell}\left(y;5,2\right)$, and compute the PDF
$p_{Z}$ of the product $Z=XY$. We illustrate the PDFs of the random
variables $X$, $Y$, and $Z$ in Figure~\ref{fig:lom-gam}.

\begin{figure}[H]
\begin{centering}
\includegraphics[scale=0.3]{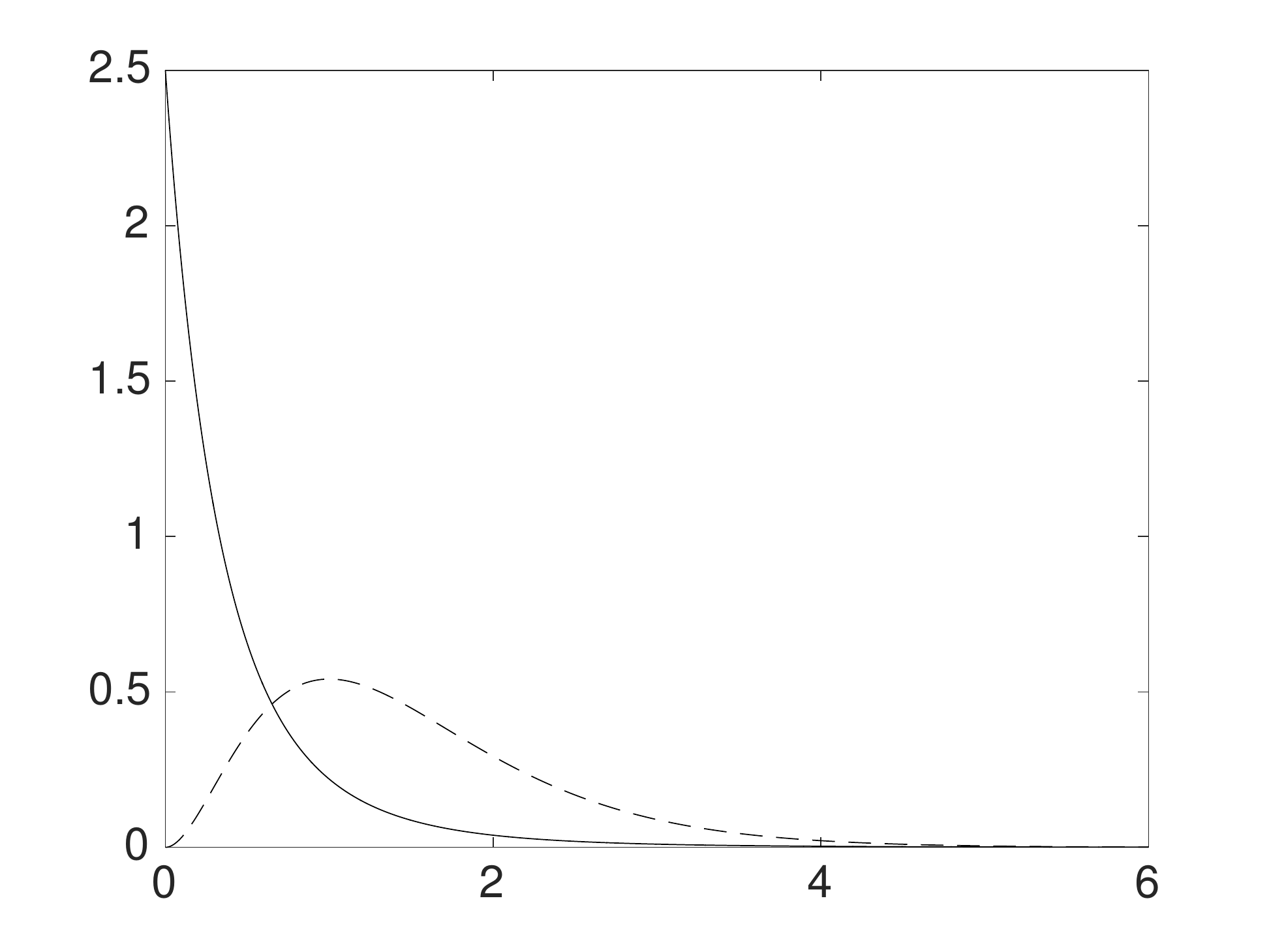}\includegraphics[scale=0.3]{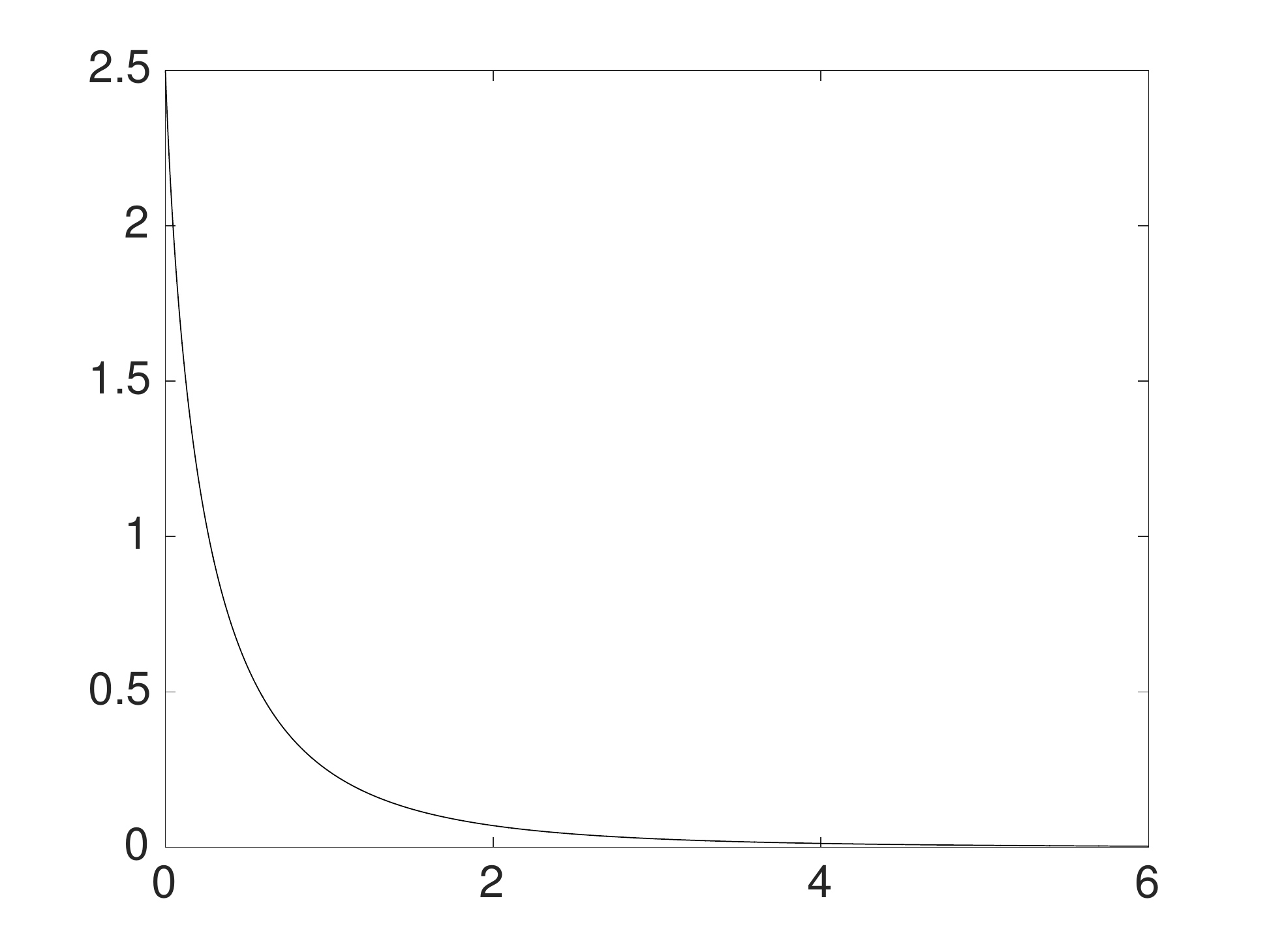}
\par\end{centering}
\centering{}\caption{\label{fig:lom-gam} PDFs of the random variables $X$ and $Y$ (left)
and the PDF of the product random variable $Z=XY$ in Example~\ref{subsec:lom-gam}
(right).}
\end{figure}

\subsubsection{\label{subsec:weibull}Product of Weibull and Nakagami random variables}

Next we consider the Weibull distributed random variable with PDF
given by 

\[
f^{w}\left(x;\lambda,k\right)=\frac{k}{\lambda}\left(\frac{x}{\lambda}\right)^{k-1}e^{-\left(\frac{x}{\lambda}\right)^{k}},\,\,\,x\geq0,
\]
where $k>0$ and $\lambda>0$ are called the shape and scale parameters.
We use a Weibull distributed random variable $X\sim f^{w}\left(x;1,1.5\right)$
and the random variable $Y$ obtained in Example~\ref{subsec:nak11}
(the product of two Nakagami random variables), and compute the PDF
$p_{Z}$ of the product $Z=XY$. We display the results in Figure~\ref{fig:weibull}.

\begin{figure}[H]
\begin{centering}
\includegraphics[scale=0.3]{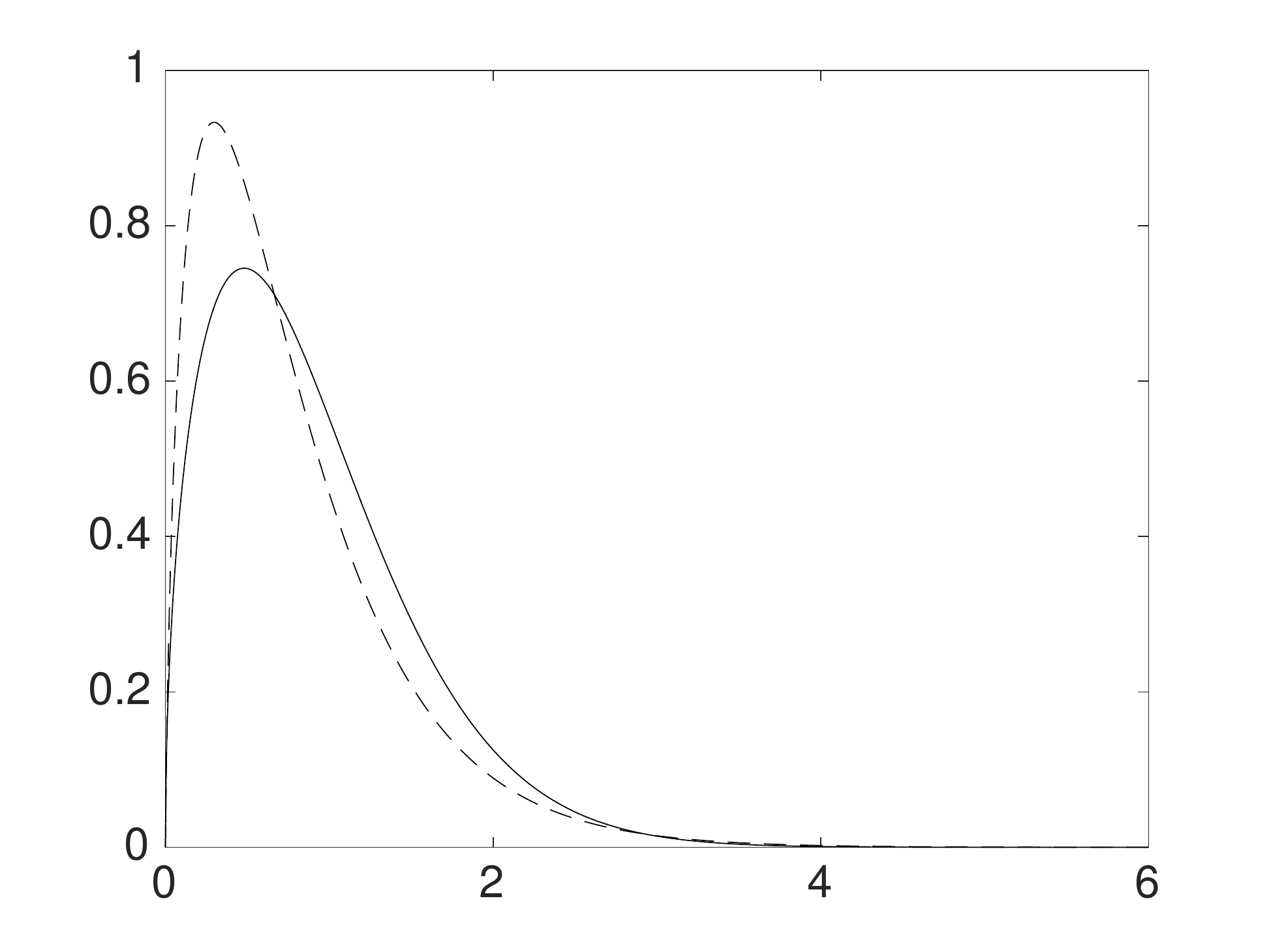}\includegraphics[scale=0.3]{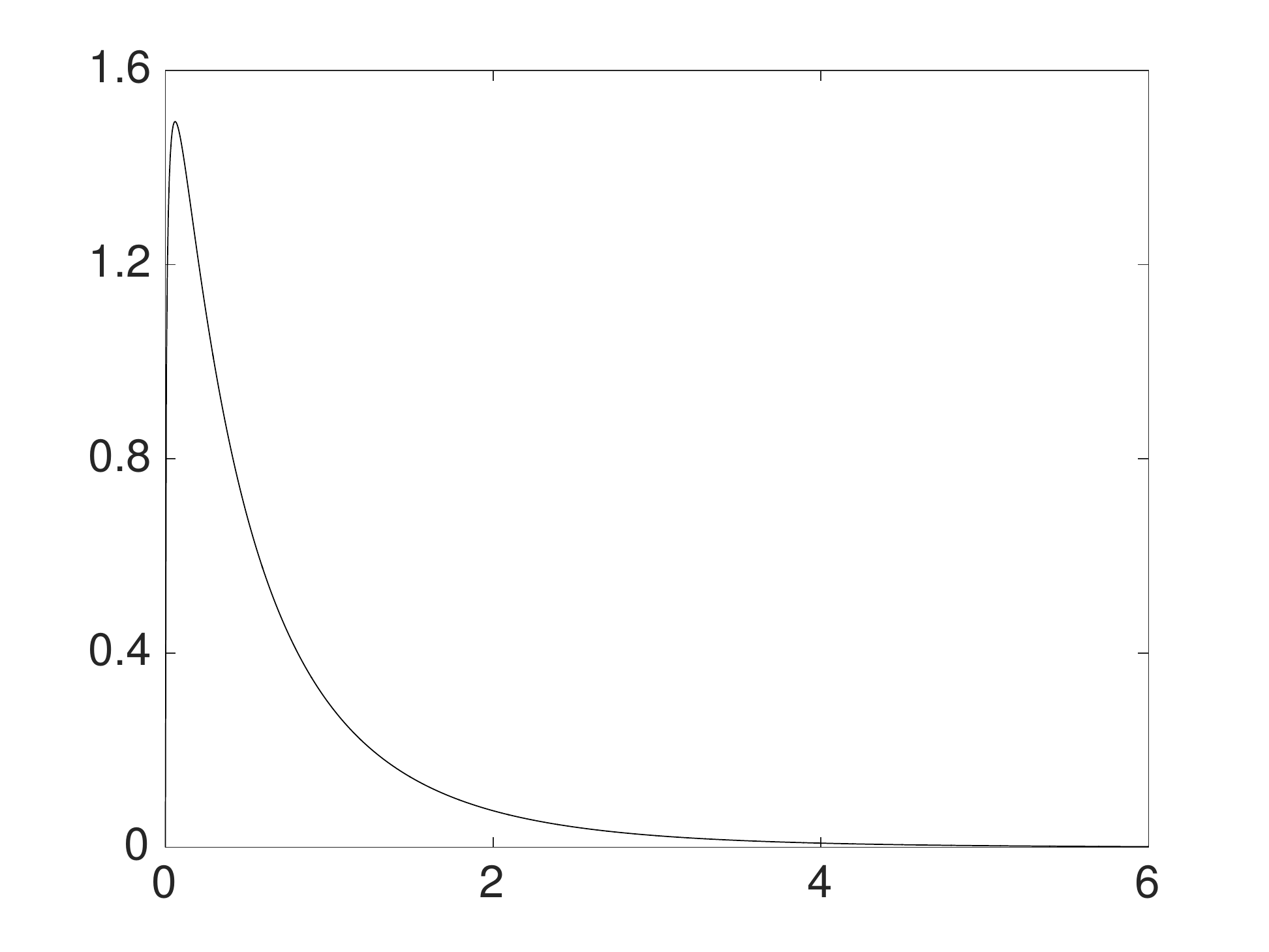}
\par\end{centering}
\centering{}\caption{\label{fig:weibull} PDFs of the random variables $X$ and $Y$ (left)
and the PDF of the product random variable $Z=XY$ in Example~\ref{subsec:weibull}
(right).}
\end{figure}

\subsubsection{\label{subsec:Quotient-of-Nakagami}Quotient of Nakagami and Gamma
random variables}

Finally, we compute the PDFs $p_{Z_{1}}$ and $p_{Z_{2}}$ of the
quotients $Z_{1}=X/Y$ and $Z_{2}=Y/X$ of Nakagami and Gamma random
variables $X$ and $Y.$ Given random variables $X\sim f^{\mathcal{N}}\left(x;1,1\right)$
as in Example~\ref{subsec:nak11} and $Y\sim f^{\gamma}\left(y,3,2\right)$
as in Example~\ref{subsec:gam205305}, we compute the PDFs of their
ratios and display the results in Figure~\ref{fig:ratio-nak-gam}.

\begin{figure}[H]
\begin{centering}
\includegraphics[scale=0.3]{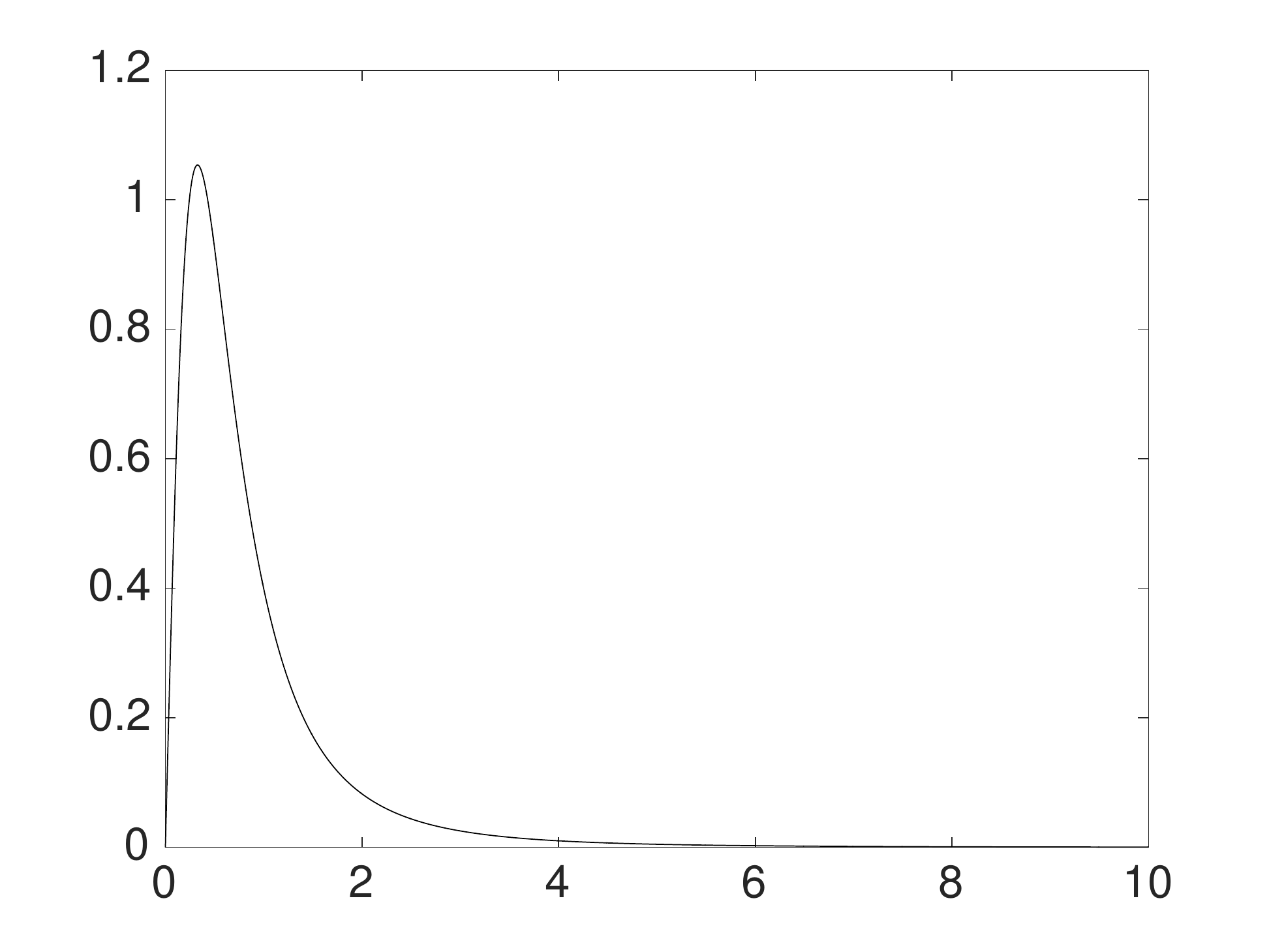}\includegraphics[scale=0.3]{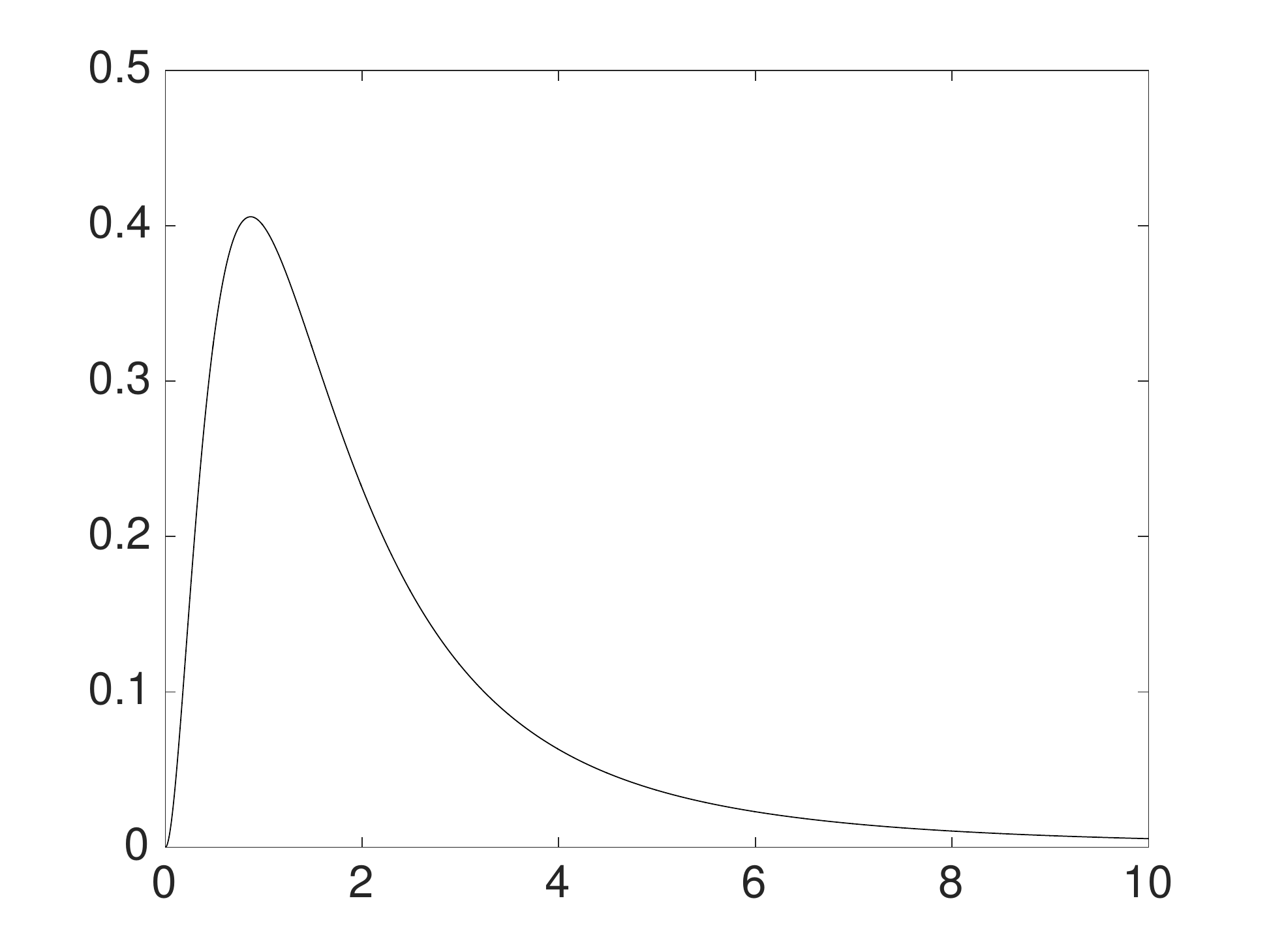}
\par\end{centering}
\centering{}\caption{\label{fig:ratio-nak-gam} PDFs of the random variables $Z_{1}=X/Y$
(left) and $Z_{2}=Y/X$ (right) in Example~\ref{subsec:Quotient-of-Nakagami}.}
\end{figure}

\subsubsection{Heavy-tailed distribution}

Our approach remains valid for heavy-tailed distributions. As an example,
let us consider a random variable $X$ with the standard Cauchy distribution,
\begin{equation}
f_{exact}\left(x\right)=\frac{1}{\pi}\frac{1}{1+x^{2}}\label{eq:Cauchy distribution}
\end{equation}
and compute the distribution for $\left|X\right|^{2}$. In this case
the integral defining the PDF of $\left|X\right|^{2}$ is evaluated
explicitly,
\begin{eqnarray}
p_{exact}\left(t\right) & = & \int_{\mathbb{R}}\int_{\mathbb{R}}f\left(x\right)f\left(y\right)\delta\left(t-\left|x\right|\left|y\right|\right)dxdy\nonumber \\
 & = & \frac{4}{\pi^{2}}\int_{0}^{\infty}\int_{0}^{\infty}\frac{1}{1+x^{2}}\frac{1}{1+y^{2}}\delta\left(t-xy\right)dxdy\nonumber \\
 & = & \frac{4}{\pi^{2}}\int_{0}^{\infty}\frac{1}{1+x^{2}}\frac{1}{t^{2}+x^{2}}xdx\nonumber \\
 & = & \frac{4}{\pi^{2}}\frac{\log\left(t\right)}{t^{2}-1},\label{eq:PDF of abs(X)^2}
\end{eqnarray}
which allows us to estimate the error of our numerical approach. We
start by approximating the Cauchy distribution (\ref{eq:Cauchy distribution})
via exponentials in the form (\ref{eq:form to maintain}). Using the
Laplace transform, we have
\begin{equation}
f_{exact}\left(x\right)=\frac{1}{\pi}\int_{0}^{\infty}e^{-x\tau}\sin\left(\tau\right)d\tau=\frac{1}{\pi}\int_{-\infty}^{\infty}e^{-xe^{s}+s}\sin\left(e^{s}\right)ds.\label{eq:integral for Cauchy distribution}
\end{equation}
Unfortunately, discretizing this integral directly via the trapezoidal
rule requires too many terms to achieve an accurate approximation
for small values of $x$. Therefore, we discretize (\ref{eq:integral for Cauchy distribution})
only to approximate the tail of (\ref{eq:Cauchy distribution}). We
obtain
\begin{equation}
\left|f_{exact}\left(x\right)-f_{tail}\left(x\right)\right|\le\epsilon,\,\,\,x\in\left[7.4748,\infty\right),\label{eq:Cauchy via exponentials}
\end{equation}
where
\[
f_{tail}\left(x\right)=h_{0}\sum_{j=M}^{N}e^{-xe^{s_{j}}+s_{j}}\sin\left(e^{s_{j}}\right)
\]
with $s_{j}=jh_{0}$, $h_{0}=0.25$, $M=-70$ and $N=5$, and achieve
accuracy $\epsilon\approx10^{-15}$. We then consider the difference
$f_{head}\left(x\right)=f_{exact}\left(x\right)-f_{tail}\left(x\right)$
in the interval $\left[0,7.4748\right]$ and use Algorithms~\ref{alg:Computing-exponential-representations-I}
or \ref{alg:Computing-exponential-representation-II} to obtain an
approximation for $f_{head}\left(x\right)$ with $15$ terms. Finally,
we remove all terms of $f_{head}\left(x\right)+f_{tail}\left(x\right)$
with weights less than $0.33\cdot10^{-12}$ leaving $73$ terms in
the resulting approximation of $f_{exact}\left(x\right)$. In Figure~\ref{fig:Log-Log-error-Cauchy}
we display the exponents and the error of the resulting approximation,
\begin{equation}
\mbox{err}_{f}\left(s\right)=\log_{10}\left(\left|f_{exact}\left(10^{s}\right)-f_{tail}\left(10^{s}\right)-f_{head}\left(10^{s}\right)\right|+10^{-20}\right),\,\,\,s\in\left[-15,6\right].\label{eq:error of approx Cauchy}
\end{equation}

\begin{figure}
\begin{centering}
\includegraphics[width=2in]{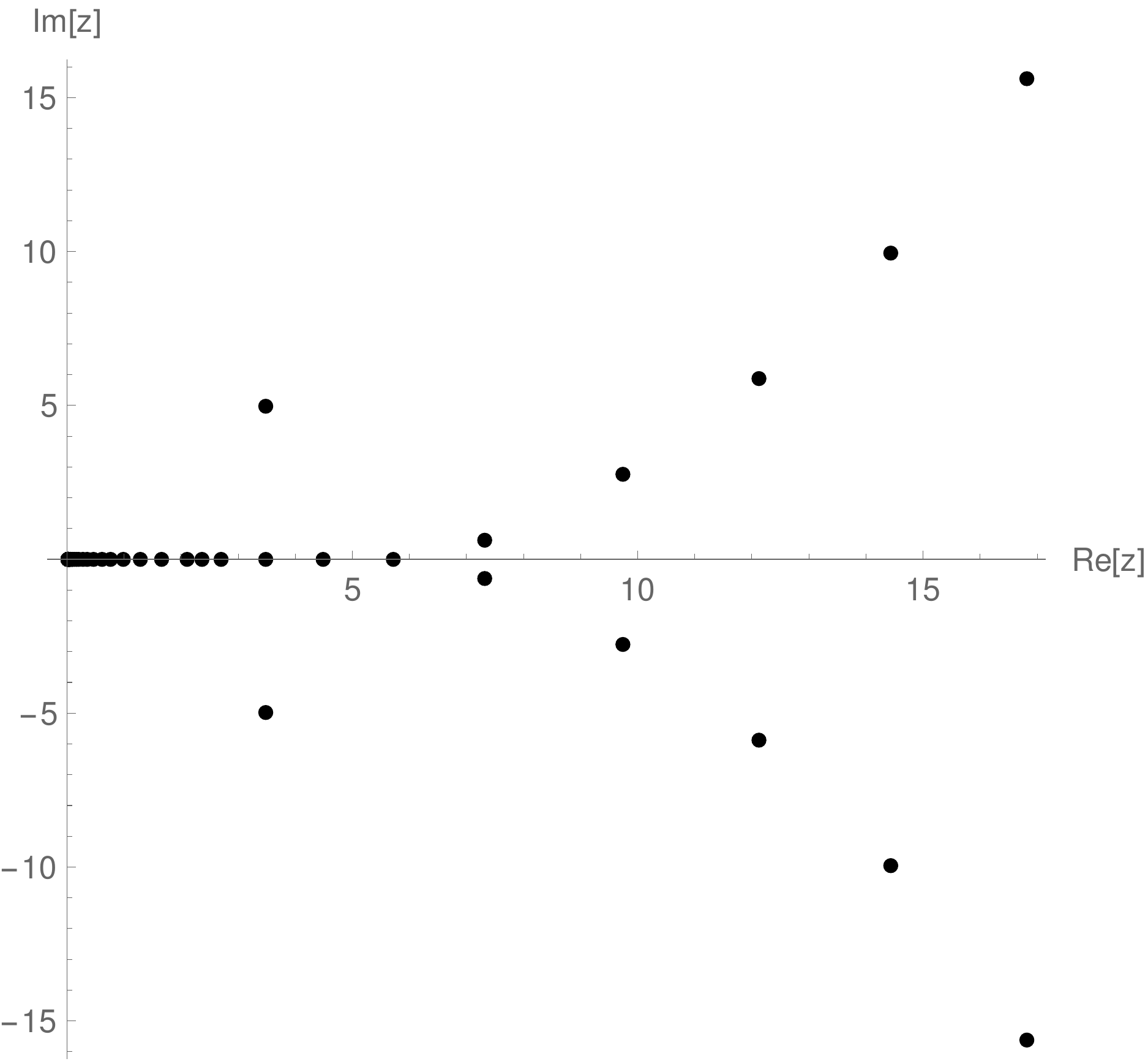}\includegraphics[width=2.5in]{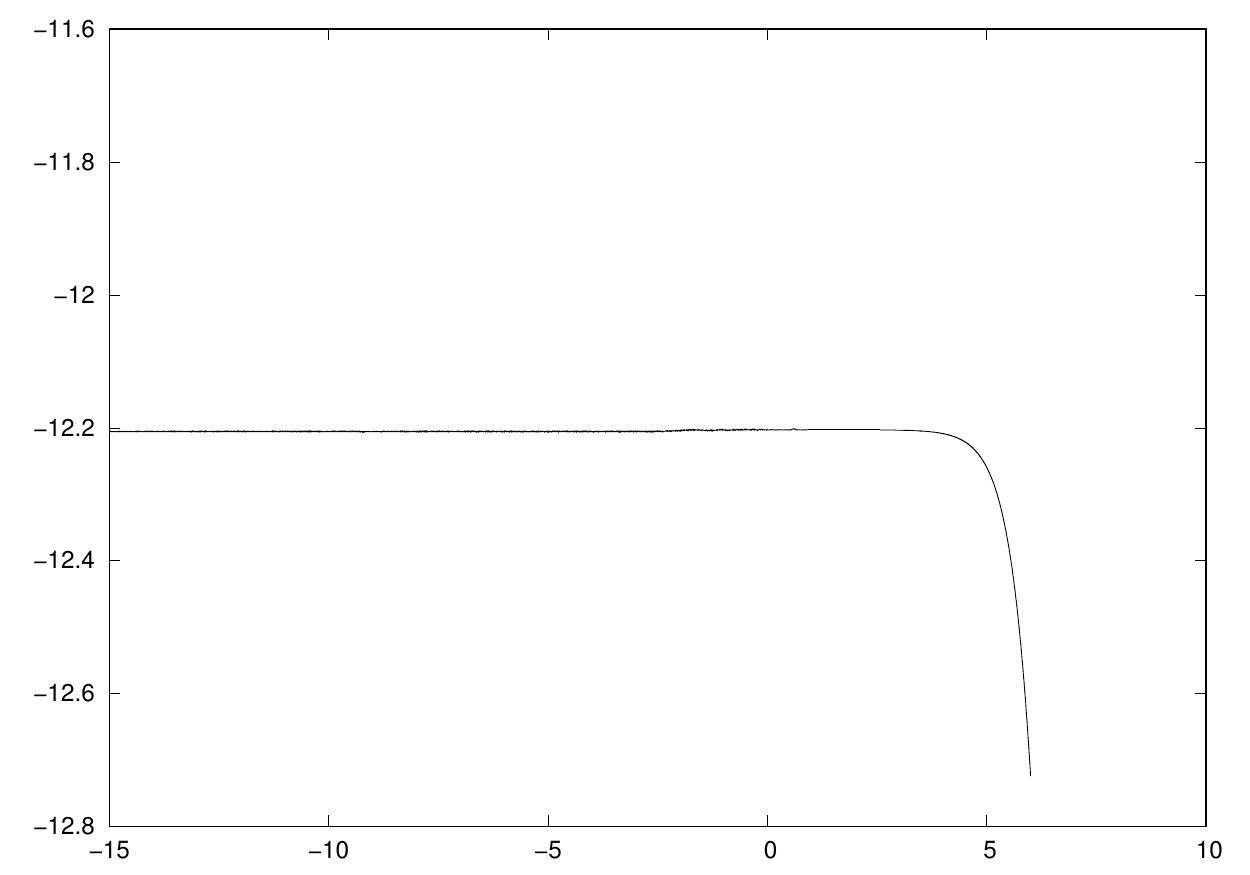}
\par\end{centering}
\caption{\label{fig:Log-Log-error-Cauchy}The $73$ exponents of the approximation
of the PDF (\ref{eq:Cauchy distribution}) (left) and log-log plot
of the resulting approximation error (\ref{eq:error of approx Cauchy})
(right). }
\end{figure}

To compute an approximation for the distribution of $\left|X\right|^{2}$,
we use $p\left(t\right)$ in Lemma~\ref{lem:2} where $\alpha=\beta=1$
and, thus,
\begin{equation}
p\left(t\right)=v\left(t\right)=2\sum_{m=1}^{M}\sum_{n=1}^{N}a_{m}b_{n}K_{0}\left(2\sqrt{t\xi_{m}\eta_{n}}\right).\label{eq:function_v for Cauchy}
\end{equation}
In order to find an exponential approximation for $p\left(t\right)$,
we first use $40,000$ equally spaced points with step size $h=2$
to discretize (\ref{eq:function_v for Cauchy}) and use Algorithms~\ref{alg:Computing-exponential-representations-I}
or \ref{alg:Computing-exponential-representation-II} to approximate
the ``tail'' of $p$. The resulting approximation with $30$ terms
is valid in the interval $\left[9.18,\infty\right)$. We then consider
the difference $p_{head}^{1}\left(t\right)=p\left(t\right)-\widetilde{p}_{tail}\left(t\right)$
in the interval $\left[10^{-4},9.18\right]$, discretize $p_{head}^{1}\left(t\right)$
in this interval using $100,000$ equally spaced points, and use one
of the mentioned algorithms to obtain an approximation for $p_{head}^{1}\left(t\right)$
with $40$ terms. We then consider $p_{head}^{2}\left(t\right)=p\left(t\right)-\widetilde{p}_{tail}\left(t\right)-p_{head}^{1}\left(t\right)$
in the interval $\left[10^{-7},0.5814\cdot10^{-3}\right]$, discretize
$p_{head}^{2}\left(t\right)$ in this interval using $50,000$ equally
spaced points, and again use Algorithms~\ref{alg:Computing-exponential-representations-I}
or \ref{alg:Computing-exponential-representation-II} to obtain an
approximation for $p_{head}^{2}\left(t\right)$ with $25$ terms.
In Figure~\ref{fig:Log-Log-error-heavy  tail} we display the error
of the final approximation with $95$ terms,
\begin{equation}
\mbox{err}_{p}\left(y\right)=\log_{10}\left(\left|p_{exact}\left(10^{y}\right)-\widetilde{p}_{tail}\left(10^{y}\right)-p_{head}^{1}\left(10^{y}\right)-p_{head}^{2}\left(10^{y}\right)\right|+10^{-20}\right),\label{eq:error of approx product}
\end{equation}
where $y\in\left[-7,8\right]$. We note that in this example we were
not seeking an optimal representation of the form (\ref{eq:form to maintain})
of the distribution for $\left|X\right|^{2}$.

\begin{figure}
\begin{centering}
\includegraphics[width=3.5in]{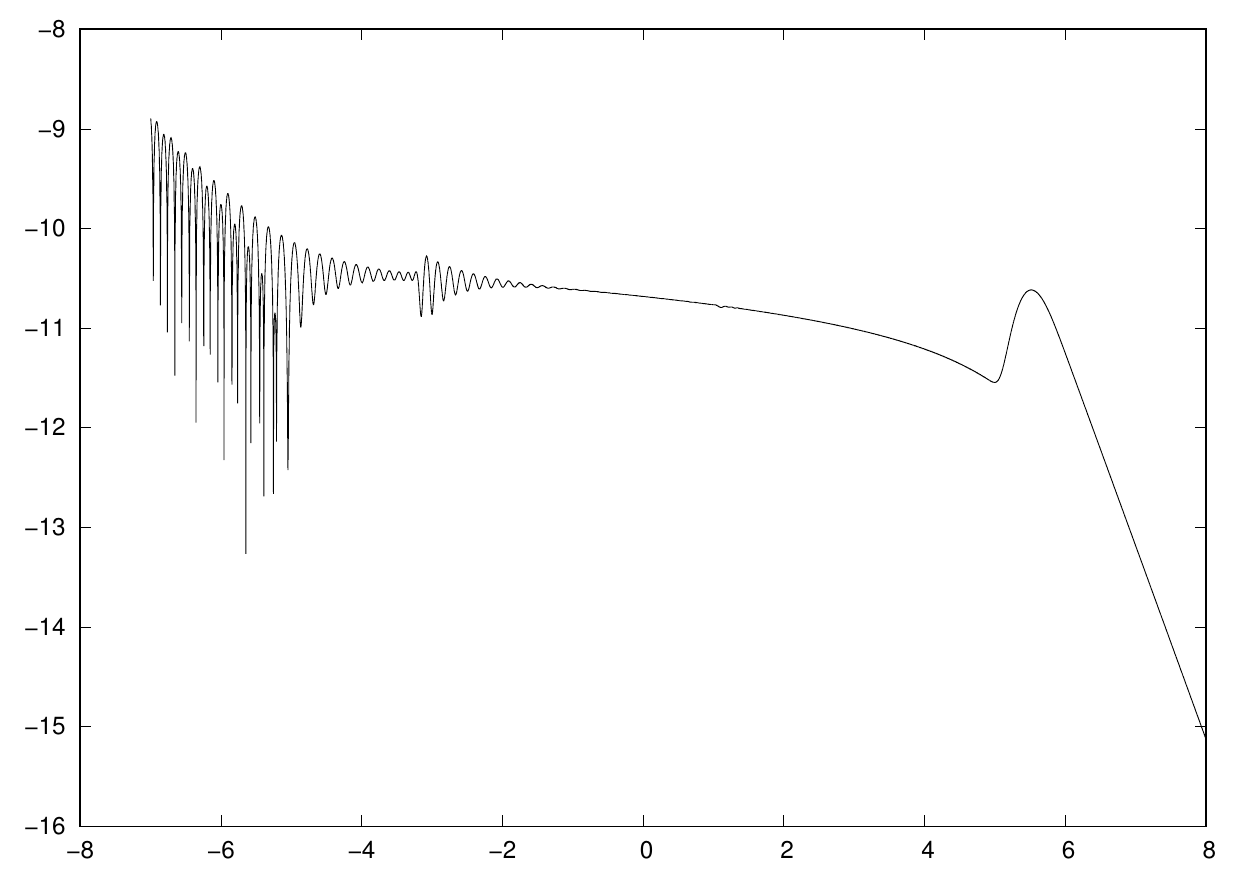}
\par\end{centering}
\caption{\label{fig:Log-Log-error-heavy  tail}Log-log plot of the error (\ref{eq:error of approx product})
of approximating the PDF (\ref{eq:PDF of abs(X)^2}) with $95$ terms.}
\end{figure}

\section{\label{sec:Computing-expectations-of}Computing expectations of functions
of random variables}

An important use of representing the PDF $p_{Z}$ of a non-negative
random variable in the proposed functional form is to compute, for
a function $u$, the expectation of the variable $u\left(Z\right)$,
\begin{equation}
\mbox{E}\left[u\left(Z\right)\right]=\int_{0}^{\infty}u\left(x\right)p_{Z}\left(x\right)dx.\label{eq:expectation integral}
\end{equation}
If the function $u$ is given analytically, i.e. we can evaluate it
at any point, the fact that we have a functional representation of
$p_{Z}$ allows us to use an appropriate quadrature to evaluate this
integral to any desired accuracy. Moreover, if the function $u$ admits
a representation via exponentials (which can be computed via Algorithms~\ref{alg:Computing-exponential-representations-I}
or \ref{alg:Computing-exponential-representation-II}), the expectation
\eqref{eq:expectation integral} can be evaluated explicitly. If only
samples of the function $u$ are provided, then we can treat $p_{Z}$
as a weight and construct a quadrature with nodes at locations where
the values of $u$ are available. If the function $u$ is a monomial,
i.e. when computing the moments of random variable $Z$, we can use
the explicit integral

\[
\int_{0}^{\infty}x^{\alpha-1}e^{-\eta x}dx=\eta^{-\alpha}\Gamma\left(\alpha\right),\ \ \mathcal{R}e\left(\eta\right)>0,\ \ \alpha>0.
\]
For example, given the PDF of a random variable $Z$ in the form \eqref{eq:form to maintain},
we compute its first moment $m_{1}$ as 

\[
m_{1}=\Gamma\left(\alpha+1\right)\sum_{l=1}^{M}a_{l}\xi_{l}^{-\left(\alpha+1\right)},
\]
and its second moment $m_{2}$ as

\[
m_{2}=\Gamma\left(\alpha+2\right)\sum_{l=1}^{M}a_{l}\xi_{l}^{-\left(\alpha+2\right)}-m_{1}.
\]
We note that while our algorithms do not guarantee that the moments
are preserved exactly, the accuracy of the resulting moments is controlled
by the overall accuracy of the approximation. In particular, we can
always enforce $\int_{0}^{\infty}p_{Z}\left(x\right)dx=1$ by imposing
an additional linear constrain on the coefficients $a_{l}$ of the
exponential approximation.

\section{\label{sec:Conclusions-and-further}Conclusions and further work }

For any user-selected accuracy, we have developed an approximate representation
of non-negative random variables \eqref{eq:form to maintain} that
allow us to compute the PDFs of their sums, products and quotients
in the same functional form. The monomial factor in the functional
form in \eqref{eq:form to maintain} is chosen to accommodate a possible
rapid change of the PDFs of non-negative random variables near zero. 

We demonstrated accuracy and efficiency of the resulting numerical
calculus of PDFs on several numerical examples. In order to account
for the boundary at zero, we use a different representation of the
PDFs of non-negative random variables than our previous construction
for random variables defined on the whole real line \cite{BE-MO-SA:2017}.
Clearly, Gaussian mixtures used in \cite{BE-MO-SA:2017} do not have
support restricted to the positive real axis and, thus, could not
yield an efficient representation. 

While there are clear advantages to our new approach for computing
PDFs in comparison with Monte Carlo-type methods, we do not compare
the two in this paper. We plan to address such comparison elsewhere
in the context of practical applications. 

\section*{Acknowledgments}

We would like to thank the anonymous reviewers for their useful comments
and Dr. Luis Tenorio (Colorado School of Mines) for his valuable suggestions.

\bibliographystyle{plain}

\end{document}